\documentclass[11pt]{article}

\usepackage{epsfig}
\usepackage{amssymb, latexsym}
\usepackage{amscd}
\usepackage[all]{xy}
\usepackage{epsf}
\usepackage[mathscr]{eucal}

\usepackage{amsmath,amsfonts,amscd,amssymb,amsthm}

\usepackage[titletoc]{appendix}
\usepackage[sc]{titlesec}

\long\def\comment#1\endcomment{}

\theoremstyle{plain}

\newtheorem{theorem}{\sc Theorem}[section]
\newtheorem{lemma}[theorem]{\sc Lemma}

\newtheorem{prop}[theorem]{\sc Proposition}
\newtheorem{coroll}[theorem]{\sc Corollary}
\newtheorem{klemma}[theorem]{\sc Key-lemma}

\theoremstyle{plain}

\theoremstyle{exercise}
\newtheorem{remark}[theorem]{\sc Remark}

\makeatletter \@addtoreset{equation}{section} \makeatother

\def\eqref#1{\thetag{\ref{#1}}}

\let\latexref=\ref
\def\ref#1{{\normalfont{\latexref{#1}}}}

\newcommand{\ldot}{{\:\raisebox{2,3pt}{\text{\circle*{1.5}}}}}
\newcommand{\udot}{{\:\raisebox{3pt}{\text{\circle*{1.5}}}}}
\def\dlim_#1{{\displaystyle\lim_{#1}}^\hdot}

\newcommand{\cchar}{\operatorname{\sf char}}

\newcommand{\Ker}{\operatorname{{\rm Ker}}}

\newcommand{\id}{\operatorname{\rm id}}

\newcommand{\Mor}{\mathrm{Mor}}
\newcommand{\Ob}{\mathrm{Ob}}

\renewcommand{\Bar}{\mathrm{Bar}}
\newcommand{\Cobar}{\mathrm{Cobar}}

\newcommand{\Hom}{\mathrm{Hom}}

\newcommand{\Hoch}{\mathrm{Hoch}}

\newcommand{\dg}{\mathrm{dg}}
\newcommand{\Ho}{\mathrm{Ho}}

\newcommand{\Sets}{\mathscr{S}ets}

\newcommand{\Vect}{\mathscr{V}ect}

\newcommand{\Alg}{{\mathscr{A}lg}}

\newcommand{\Cat}{{\mathscr{C}at}}

\newcommand{\Fun}{{\mathrm{Fun}}}

\renewcommand{\k}{\Bbbk}

\renewcommand{\cchar}{\mathrm{char}\ }

\newcommand{\coh}{\mathrm{coh}}
\newcommand{\Coh}{{\mathscr{C}{oh}}}

\newcommand{\sotimes}{\overset{\sim}{\otimes}}

\newcommand{\lin}{\mathrm{lin}}

\newcommand{\Hot}{\mathrm{Hot}}

\newcommand{\FFun}{{\mathscr{F}un}}
\newcommand{\Tot}{\mathrm{Tot}}
\newcommand{\RRHom}{{\mathbb{R}\underline{\mathrm{Hom}}}}
\newcommand{\Conv}{\mathrm{Conv}}

\newcommand{\g}{\mathfrak{g}}
\newcommand{\MC}{\mathrm{MC}}
\newcommand{\sm}{\mathrm{sf}}

\title{\sc{On the twisted tensor product of small dg categories}}

\author{\sc{Boris Shoikhet}}
\date{}

\begin{document}\maketitle
{\footnotesize
\begin{center}{\parbox{4,5in}{{\sc Abstract.}
Given two small dg categories $C,D$, defined over a field, we introduce their (non-symmetric) twisted tensor product $C\sotimes D$. We show that $-\sotimes D$ is left adjoint to the functor $\Coh(D,-)$, where $\Coh(D,E)$ is the dg category of dg functors $D\to E$ and their coherent natural transformations. This adjunction holds in the category of small dg categories (not in the homotopy category of dg categories $\Hot$). We show that for $C,D$ cofibrant, the adjunction descends to the corresponding adjunction in the homotopy category. Then comparison with a result of To\"{e}n [To] shows that, for $C,D$ cofibtant, $C\sotimes D$ is isomorphic to $C\otimes D$, as an object of the homotopy category $\Hot$.

}}
\end{center}
}

\section*{\sc Introduction}

\subsection{\sc }
The construction of {\it twisted tensor product} $C\sotimes D$ of small dg categories, provided in this paper, mimics for the case of dg enrichment the construction of {\it Gray product} of strict 2-categories [Gr]. Recall that for the Gray product $A\times_G B$ of strict 2-categories the diagram
\begin{equation}\label{eqintro1}
\xymatrix{
a\times b\ar[r]^{f\times\id}\ar[d]_{\id\times g}&    a_1\times b\ar[d]^{\id\times g}\\
a\times b_1\ar[r]_{f\times \id}\ar@2{->}[ru]&a_1\times b_1
}
\end{equation}
commutes only up to a new 2-arrow. It defines a monoidal product on the category $\Cat_2$ of strict 2-categories, and there is an 
adjunction
\begin{equation}\label{eqintro2}
\Cat_2(A\times_G B,C)=\Cat_2(A,{\mathscr{P}s}(B,C))
\end{equation}
where ${\mathscr{P}s}(-,-)$ stands for the 2-category of {\it pseudo-natural transformations}.
There is a monoidal closed model structure on the category $\Cat_2$ [L], and from this point of view the Gray product is better than the cartesian product. 

Let $C,D$ be small dg categories over a field $\k$. We construct a small dg category $C\sotimes D$ as follows. It has objects $\Ob(C)\times \Ob(D)$, and its morphisms is a free dg envelope of morphism in $C\times\id_Y$, $\id_X\times D$, for any $X\in \Ob(C), Y\in\Ob(D)$, and of some new morphisms. 

The simplest among the new morphism is $\varepsilon(f;g)$ of degree $\deg\varepsilon(f;g)=\deg f+\deg g-1$, for homogeneous morphisms $f$ in $C$ and $g$ in $D$, $f\colon X_0\to X_1, g\colon Y_0\to Y_1$. We define its  differential as
\begin{equation}\label{eqintro3}
\begin{aligned}
\ &d\varepsilon(f;g)+\varepsilon(df;g)+(-1)^{\deg f}\varepsilon(f;dg)=\\
&(f\otimes \id_{Y_1})\star (\id_{X_0}\otimes g)-
(-1)^{\deg f(\deg g+1)}(\id_{X_1}\otimes g)\star 
(f\otimes \id_{Y_0})
\end{aligned}
\end{equation}
in the spirit of the Drinfeld dg quotient [Dr].

This new morphism $\varepsilon(f;g)$ is thought of as a dg counterpart of the new 2-morphism in diagram \eqref{eqintro1} for the case of Gray product.
There are also higher new morphisms $\varepsilon(f;\ g_1,\dots,g_n)$, for a homogeneous morphism $f$ in $C$, and for a chain of composable homogeneous morphisms $g_1,\dots,g_n$ in $D$, with the differentials of them defined accordingly, see \eqref{eqd1}. They are subject to relations $(R_1)$-$(R_4)$, see Section \ref{sectiondeftwist}

Our first result shows that an adjunction, analogous to \eqref{eqintro2}, holds.

\begin{theorem}\label{theorem01}
Let $C,D,E$ be small dg categories over a field $\k$. One has an adjunction
\begin{equation}\label{eqintro4}
\Cat_\dg(C\sotimes D,E)=\Cat_\dg(C,\Coh_\dg(D,E))
\end{equation}
where $\Coh_\dg(-,-)$ stands for the dg category, whose objects are dg functors and whose morphisms are coherent natural transformations between them.
\end{theorem}
We stress that this adjunction holds in the category $\Cat_\dg(\k)$ itself, not in its homotopy category.
\subsection{\sc }
The most of work in this paper is done for computation of the homotopy type of the dg category $C\sotimes D$.
We are able to find the homotopy type of $C\sotimes D$, provided $C,D$ are cofibrant for the Tabuada closed model structure on $\Cat_\dg$ [Tab1]. We have:

\begin{theorem}\label{theorem02}
Let small dg categories $C,D$ be cofibrant for the Tabuada closed model structure. Then $C\sotimes D$ is isomorphic, as an object of the homotopy category $\Hot$ of small dg categories, to the ordinary tensor product $C\otimes D$. 
\end{theorem}

Let us outline the main steps in the proof of Theorem \ref{theorem02}. 

The idea is to show that, for cofibrant $C,D$, \eqref{eqintro4} descends to an adjunction in the homotopy category $\Hot$:
\begin{equation}\label{eqintro5}
\Hot(C\sotimes D,E)=\Hot(C,\Coh_\dg(D,E))
\end{equation}
By a result of Faonte [Fa2], $\Coh_{A_\infty}(-,-)$ has the homotopy type of the derived Hom $\RRHom(-,-)$ of two small dg categories, introduced in [To], and for $D$ cofibrant, the dg categories $\Coh_{A_\infty}(D,E)$ and $\Coh_\dg(D,E)$ are isomorphic in $\Hot$. Then Theorem \ref{theorem02} follows from \eqref{eqintro5} and the fundamental result of To\"{e}n [To, Cor. 6.4] saying that $-\otimes D$ is the left adjoint 
to $\RRHom(D,-)$ in $\Hot$.

The most non-trivial part is to show that \eqref{eqintro4} descends to \eqref{eqintro5}. Our strategy is as follows.

We compute the homotopy relation in $\Cat_\dg(C,D)$, for $C$ cofibrant, as the right cylinder homotopy relation, with the path object $\hat{D}$ of $D$, introduced in [Tab2]. 
The adjunction \eqref{eqintro4} gives
\begin{equation}
\Cat_\dg(C\sotimes D,\hat{E})=\Cat_\dg(C,\Coh_\dg(D,\hat{E}))
\end{equation}
It is quite clear that $C\sotimes D$ is cofibrant if $C$ and $D$ are. 
Then everything reduces to
\begin{theorem}\label{theorem03}
Let $D,E$ be small dg categories, with $D$ cofibrant. Then $\Coh_\dg(D,\hat{E})$ is a path object of the dg category $\Coh_\dg(D,E)$.
\end{theorem}
Theorem \ref{theorem03} may have an independent interest. We know from [Tab2] that $\widehat{\Coh_\dg(D,E)}$ is a path object, and one needs to prove the same thing for $\Coh_\dg(D,\hat{E})$. To prove it, we revisit the proof that $\hat{C}$ is a path object of $C$, given in [Tab2, Prop. 2.0.11]. We replace it by a more direct argument, which works as well in the refined situation. First of all, we consider the case when $D$ is an $I$-cell complex\footnote{The concepts of a relative $I$-cell complex and of a $I$-cell complex are the counterparts, for the case of dg categories, of the concepts of a relative semi-free map of dg algebras and of a semi-free dg algebra. See Section \ref{sectioncmcdg} for the definition.}. It is the most tricky part of the paper, see Key-Lemma \ref{keylemma}. 

\begin{remark}{\rm
Note that for the case of Gray product we started our discussion with it is true that for {\it any} two 2-categories $C,D$ the natural projection $C\times_G D\to C\times D$ is a weak equivalence, so these two 2-categories are isomorphic as objects of the homotopy category, see [L, Section 2].\\
In our situation, a similar property would seemingly hold if we considered an ``$A_\infty$ version'' $C\sotimes_\infty D$ of the category $C\sotimes D$. For this $A_\infty$ version $C\sotimes_\infty D$, the higher by $f$ elements $\varepsilon(f_1,\dots,f_m;\ g_1,\dots,g_n)$ are also added, and \eqref{eqsuper} is replaced by a sequence of higher $A_\infty$ equations.
The category $C\sotimes_\infty D$ has the following drawback: the adjunction \eqref{eqintro4} fails, and it is not replaced by any other adjunction (at least, outside of the world of $\infty$-categories).
}
\end{remark}
\subsection{\sc }
\subsubsection{}
This paper is the first one in a bigger project. 
In our next paper(s), we establish some associativity properties for $C\sotimes D$, and apply it to a construction of a contractible 2-operad in the sense of Batanin [Ba3,4,5], acting on the category of small dg categories. This contractible 2-operad is different from the one found in [T2], and has some fruiful applications. \\
The statement of Theorem \ref{theorem02} is needed to show that the constructed 2-operad is contractible.
\subsubsection{}
 Another, more general and less explicit, approach to the Gray-like product for different enrichments has been developed in [Ba1,2], [St]. 
\subsection{\sc }
The paper is organized as follows.

In Section 1, we recall the definition and some basic facts on coherent natural transformations.

In Section 2 we introduce our main object of study here: the twisted tensor products of small dg categories $C\sotimes D$. Theorem \ref{theorem01} is proved as Theorem \ref{prop1}. 

In Section 3 we recall some facts from closed model categories necessary for proofs of Theorem \ref{theorem02} and Theorem \ref{theorem03}. It is standard. The only exception is a more direct proof of a result of Tabuada [Tab2] on a path object of a dg category, given in Lemma \ref{lemmaf2}. We need this direct proof for a proof of Key-Lemma \ref{keylemma}, more specifically, for Lemma \ref{lemmax}.

Section 4 is the technical core of the paper. It is devoted to a proof of Theorem \ref{theorem02}, which figures as Theorem \ref{theorht}.  Theorem \ref{theorem03} is proved in Proposition \ref{propstep1}, with the most essential step made in Key-Lemma \ref{keylemma}.

\subsection*{\sc }

\subsection*{\sc Acknowledgements}
I am thankful to Michael Batanin for introducing me to the Gray product of 2-categories and for providing me with several related references. \\
I am thankful to Bernhard Keller for introducing me to dg categories (some 10 years ago) and for providing me with several references.\\
The work was partially supported by the FWO Research Project Nr. G060118N and 
by the Russian Academic Excellence Project `5-100'.

\vspace{2cm}

\section{\sc Coherent natural transformations}

\subsection{\sc Notations}
Throughout the paper, $\k$ denotes a field of any characteristic. 
For a graded vector space, we denote by $|v|$ the degree of a homogeneous element $v\in V$.

We denote by $C,D,E,\dots$ small dg categories over $\k$ (see [K2]).
The set of dg functors $F\colon C\to D$ is denoted by $\Fun_\dg(C,D)$. The set of $A_\infty$ functors $F\colon C\to D$ is denoted by $\Fun_{A_\infty}(C,D)$. The ordinary category whose objects are small dg categories over $\k$ and whose morphisms are dg functors is denoted by $\Cat_\dg(\k)$. 

For a category $\mathscr{C}$, we sometimes use the notation ``$X\in\mathscr{C}$'' meaning that $X$ is an object of $\mathscr{C}$.

For an abelian category $\mathscr{A}$, we denote by $\mathscr{A}^\udot$ the dg category whose objects are complexes of objects of $\mathscr{A}$, and whose $\Hom$'s are defined as the $\Hom$'s of the underlying graded abelian groups, with the differential acting on it as $d(\phi)(x)=d(\phi(x))-(-1)^{|\phi|}\phi(dx)$.

With a dg category $C$ over $\k$ is associated a $\k$-linear category $H^0(C)$. It has the same objects as $C$, and $H^0(C)(x,y)=H^0(C(x,y))$.

A dg functor $F\colon C\to D$ is called a {\it quasi-equivalence} if (a) the map $F\colon C(x,y)\to D(F(x),F(y))$ is a quasi-isomorphism of complexes, and (b) the corresponding functor $H^0(F)\colon H^0(C)\to H^0(D)$ is an equivalence of $\k$-linear categories. 

The category $\Cat_\dg(\k)$ of small dg categories over $\k$ admits a closed model structure, whose weak equivalences are quasi-equivalences. It is due to Tabuada [Tab1]. We discuss this closed model structure in more detail in Section \ref{sectioncmcdg} below.

\subsection{\sc The definition}
We recall the definition of a {\it coherent natural transformation} $F\Rightarrow G\colon C\to D$, where $C,D$ are small dg categories over $\k$, and $F,G$ are dg (resp., $A_\infty$) functors $C\to D$. 

Let $C,D\in \Cat_\dg(\k)$, and let $F,G\colon C\to D$ be dg functors. Associate with $(F,G)$ a cosimplicial set $\coh_\ldot(F,G)$, as follows.

Set $$\coh_0(F,G)=\prod_{X\in C}\Hom_D(F(X),G(X))$$ and 
\begin{equation}
\coh_n(F,G)=\prod_{X_0,X_1,\dots,X_n\in C}\underline{\Hom}_\k\Big(C(X_{n-1},X_n)\otimes \dots\otimes C(X_0,X_1),\ D(F(X_0),G(X_n)\Big)
\end{equation}
where $\underline{\Hom}_\k$ stands for the enriched over dg vector spaces inner Hom.

The coface maps 
$$
d_0,\dots,d_{n+1}\colon \coh_n(F,G)\to\coh_{n+1}(F,G)
$$
and the codegeneracy maps
$$
\eta_0,\dots,\eta_n\colon \coh_{n+1}(F,G)\to\coh_n(F,G)
$$
are defined in the standard way, see e.g. [T2, Sect. 3].

For example, recall the coface maps $d_0,d_1,d_2\colon \coh_1(F,G)\to \coh_2(F,G)$. For $$\Psi\in \prod_{X_0,X_1\in C}
\underline{\Hom}_\k\big(\Hom_C(X_0,X_1),\ \Hom_D(F(X_0),G(X_1)\big)
$$
one has:
\begin{equation}
\begin{aligned}
\ &d_0(\Psi)(X_0\xrightarrow{f}X_1\xrightarrow{g}X_2)=G(X_1\xrightarrow{g}X_2)\circ \Psi(X_0\xrightarrow{f}X_1)\\
&d_1(\Psi)(X_0\xrightarrow{f}X_1\xrightarrow{g}X_2)=\Psi(X_0\xrightarrow{gf}X_2)\\
&d_2(\Psi)(X_0\xrightarrow{f}X_1\xrightarrow{g}X_2)=\Psi(X_1\xrightarrow{g}X_2)\circ F(X_0\xrightarrow{f}X_1)
\end{aligned}
\end{equation}

We set:
\begin{equation}
\Coh(F,G)=\int_{[n]\in\Delta}\underline{\Hom}_{\k}(C_\udot(\k\Delta(-,[n])),\coh_n(F,G))
\end{equation}
where $\k\Delta(-,n)$ is a simplicial vector space, and $C_\udot(-)$ stands for its normalised Moore complex, and $\int_{-}$ denotes the end.
One easily shows that the complex $\Coh(F,G)$ is isomorphic, up to signs in the differential, to the product-total complex of the cochain complex $\Tot_\Pi(C^\udot(\coh_\udot(F,G)))$.

Alternatively, $\Coh(F,G)$ can be defined as the Hochschild cochain complex of $C$ with coefficients in the $C$-bimodule $\Hom_D(F(-),G(-))$.

The definition of $\Coh(F,G)$ can be upgraded for the case when $F,G\colon C\to D$ are $A_\infty$ functors, see [LH, Ch. 8]. 

One defines two dg categories, associated with a pair $C,D$ of small dg categories over $\k$, $\Coh_\dg(C,D)$ and $\Coh_{A_\infty}(C,D)$.

The dg category $\Coh_\dg(C,D)$ has the dg functors $F\colon C\to D$ as its objects, and
$$
\Coh_\dg(C,D)(F,G):=\Coh(F,G)
$$
as its $\Hom$-complexes.

The dg category $\Coh_{A_\infty}(C,D)$ has the $A_\infty$ functors $C\to D$ as its objects and
$$
\Coh_{A_\infty}(C,D)(F,G):=\Coh(F,G)
$$
as its $\Hom$-complexes.

The construction of $\Coh_*(C,D)$ is functorial with respect to dg (corresp.) $A_\infty$ functors $f\colon C_1\to C$ and $g\colon D\to D_1$, and gives rise to dg (corresp., $A_\infty$) functors
$$
f^*\colon\Coh_*(C,D)\to \Coh_*(C_1,D),\ \ g_*\colon \Coh_*(C,D)\to\Coh_*(C,D_1)
$$
where $*=\dg$ (corresp., $*=A_\infty$).

The following result has a fundamental value:
\begin{prop}\label{fquis}
Let dg functors $f\colon C_1\to C$ and $g\colon D\to D_1$ be quasi-equivalences. Then the dg functors $f^*\colon \Coh_{A_\infty}(C,D)\to\Coh_{A_\infty}(C_1,D)$ and $g_*\colon \Coh_{A_\infty}(C,D)\to\Coh_{A_\infty}(C,D_1)$ are quasi-equivalences. 
\end{prop}
It is proven in [LH, Ch.8] that $\Coh_{A_\infty}(C,D)$ is bi-functorial with respect to the $A_\infty$ functors. It follows from [LH, Theorem 9.2.0.4] that a weak equivalence can be inverted as an $A_\infty$ functor. The statement follows from these two results.

\qed

\subsection{\sc An adjunction}
The category $\Cat_\dg(\k)$ admits a Quillen closed model structure, whose weak equivalences are quasi-equivalences of dg categories  [Tab1]. 
We recall fibrations and cofibrations of this closed model structure in Section \ref{sectioncmcdg} below.

Denote by $\Hot$ the homotopy category of dg categories, which is defined as the localization of $Cat^\dg(\k)$ by weak equivalences.

B. To\"{e}n proved that $\Hot$ is a symmetric closed category, whose external Hom is denoted by $\RRHom(C,D)$. It is a dg category whose objects are {\it quasi-functors} and whose morphisms are their derived maps, see [To].
For two dg categories $C,D$ over $\k$, denote by $C\otimes D$ their {\it tensor product} over $\k$. Its objects are $\Ob(C)\times \Ob(D)$, and 
$$
\Hom_{C\otimes D}(x\times x_1, y\times y_1):=\Hom_C(x,y)\otimes_\k\Hom_D(x_1,y_1)
$$

The following fundamental adjunction is proven in [To, Cor. 6.4]:
\begin{theorem}\label{theortoen}
For $C,D,E$ in $\Cat_\dg(\k)$, one has
\begin{equation}
\Hot(C\otimes D,E)=\Hot(C,\RRHom(D,E))
\end{equation}
\end{theorem}

\qed

Faonte proved in [Fa2, Th. 1.7] the following result, linking Theorem \ref{theortoen} with $\Coh(C,D)$:
\begin{theorem}\label{theoremfa}
For two small dg categories $C,D$ over $\k$, there is an isomorphism in $\Hot$:
\begin{equation}
\RRHom(C,D)\simeq\Coh_{A_\infty}(C,D)
\end{equation}
Consequently, one has:
\begin{equation}\label{adjtf}
\Hot(C\otimes D,E)=\Hot(C,\Coh_{A_\infty}(D,E))
\end{equation}
\end{theorem}

\qed

Our first result ``refines'' the adjunction \eqref{adjtf} to the situation when the category $\Hot$ is replaced by $\Cat_\dg(\k)$.
Namely, we provide a construction of a dg category $C\sotimes D$, called {\it the twisted tensor product} of $C$ and $D$, such that the following adjunction holds:
\begin{equation}\label{adjb}
\Fun_\dg(C\sotimes D,E)=\Fun_\dg(C,\Coh_\dg(D,E))
\end{equation}
\comment
Note an essential difference in the categories $\Coh_*(D,E)$ in \eqref{adjtf} and in \eqref{adjb}:
there stands $\Coh_{A_\infty}$ in \eqref{adjtf}, and $\Coh_\dg$ in \eqref{adjb}.

Thus one should impose some conditions which guarantee that the two $\Coh_*(D,E)$ are isomorphic, as objects of $\Hot$. 
\endcomment

\comment
\subsection{\sc The functor $L$}\label{sectionl}
Here we construct a special cofibrant replacement functor $L\colon\Cat^\dg(\k)\to\Cat^\dg(\k)$, which plays an important role in the paper. It is a dg categorical version of the (simplicial) dg nerve construction, see ???. 

Let $C$ be a small dg category with a set of objects $S$. 
The small category $L(C)$ has the set of objects $S$, and the morphisms are defined as follows.

The underlying graded category of $L(C)$ is the free $\k$-linear category generated by the following graph $\Gamma(C)$.
The graph $\Gamma(C)$ contains a morphism $c(f)$ of degree $|f|$, for each homogeneous morphism $f$ in $C$, a morphism $c({f_1,f_2})\in \Hom^{|f_1|+|f_2|-1}(X_0,X_2)$, for each pair of homogeneous morphisms $f_0\in\Hom(X_0,X_1)$ and $f_1\in\Hom(X_1,X_2)$, a morphism $c({f_1,f_2,f_3})\in \Hom^{|f_1|+|f_2|+|f_3|-2}(X_0,X_3)$, for all homogeneous $f_1\in\Hom(X_0,X_1)$, $f_2\in\Hom(X_1,X_2)$, $f_3\in \Hom(X_2,X_3)$, and so on. 

For $n$ composable homogeneous morphisms $f_1\in \Hom^{|f_1|}(X_0,X_1),\dots,f_n\in\Hom^{|f_n|}(X_{n-1},X_n)$, $\Gamma(C)$ contains a homogeneous morphism $c(f_1,\dots,f_n)\in\Hom^{|f_1|+\dots+|f_n|-n+1}(X_0,X_n)$.
The morphisms $c(f_1,\dots,f_n)$ are supposed to be $\k$-linear in each argument.

The differential $d(c(f)):=c(d(f))$, and for $n\ge 2$:
\begin{equation}\label{ldif}
\begin{aligned}
\ &d(c({f_1,\dots,f_n})):=\sum_i\pm c({f_{i+1},\dots,f_n})\cdot c({f_1,\dots,f_i})   \pm\\
&\sum \pm c({f_1,\dots,f_{i-1},\  f_{i+1}\cdot f_i,\  f_{i+2},\dots, f_n})+\sum\pm c({f_1,\dots,df_i, \dots,f_n})
\end{aligned}
\end{equation}

 For $n=2$, \eqref{ldif} reads:
 \begin{equation}
 d(c(f_1,f_2))+c(df_1,f_2)+(-1)^{|f_1|}c(f_1,df_2)=c(f_2)c(f_1)-c(f_2f_1)
 \end{equation}

The dg category $L(C)$ is defined as the free dg category whose objects are $\Ob(C)$, whose morphisms are freely generated by the graph $\Gamma(C)$ described above, and with the differential of the generators given in \eqref{ldif}.

Consider the ordinary category $\Cat_\dg_\infty(\k)$ whose objects are small dg categories ober $\k$ and whose $\Hom$-sets are all $A_\infty$-functors. The set of all $A_\infty$ functors is denoted by $\Fun_{A_\infty}(-,-)$.

One has:
\begin{lemma}
The assignment $C\rightsquigarrow L(C)$ gives rise to a functor $L\colon \Cat_\dg_\infty(\k)\to\Cat^\dg(\k)$. It is left adjoint to the ``forgetful'' functor $R\colon \Cat_\dg(\k)\to\Cat^\dg_\infty(\k)$:
\begin{equation}\label{adjl}
\Fun_\dg(L(C),D)=\Fun_{A_\infty}(C,D)
\end{equation}
\end{lemma}
\qed

As usual, the pair $(L,R)$ of adjoint functors gives rise to a comonad $T=LR\colon \Cat_\dg(\k)\to\Cat^\dg(\k)$, where $T(C)=L(C)$.
One gets dg functors $\varkappa\colon L(C)\to C$  and $\mu\colon L(C)\to L(L(C))$.

\begin{lemma}
For any small dg category $C$, the dg category $L(C)$ is cofibrant, and the 
monadic maps $\varkappa\colon L(C)\to C$ and $\mu\colon L(C)\to L(L(C))$ are quasi-equivalences. 
\end{lemma}
\qed

One gets a dg functor, given on objects by the isomorphism of \eqref{adjl}, and on morphisms induced by $\varkappa\colon L(C)\to C$:
\begin{equation}
F_\varkappa\colon \Coh_\dg(L(C),D)\to \Coh_{A_\infty}(C,D)
\end{equation}
\begin{lemma}
Let $C$ be a small dg category over $\k$. Then the dg functor $F_\varkappa$ is a quasi-equivalence.
\end{lemma}
\qed

One has:
\begin{prop}\label{propr}
Let $C,D$ be small dg categories over a field $\k$, $C$ is cofibrant. Then
the natural embedding
\begin{equation}
\Coh_\dg(C,D)\hookrightarrow\Coh_{A_\infty}(C,D)
\end{equation}
is a quasi-eaquivalence of dg categories.
\end{prop}
\begin{proof}
{\it (i):}

{\it (ii):}
One has:
$$
\xymatrix{
&\Coh_\dg(L(C),D)\ar[dl]_{\sim}\ar[dr]^{F_\varkappa}\\
\Coh_{A_\infty}(L(C),D)\ar[rr]&&\Coh_{A_\infty}(C,D)
}
$$
where the left angle arrow is a weak equivalence by (i), and $F_\varkappa$ is a weak equivalence by Lemma ???. Now the statement follows from Proposition \ref{propfaonte}(i), and from the Yoneda lemma applied to $\Hot(\Cat_\dg(\k))$.

\end{proof}

\begin{coroll}
Let $D$ be a cofibrant small dg category over $\k$. Then the functor $\Coh_\dg(D,-)$ sends weak equivalences to weak equivalences. 
In particular, for any dg category $C$, the functor $\Coh_\dg(L(C),-)$ sends weak equivalences to weak equivalences. 
\end{coroll}

\qed

\endcomment

\section{\sc The twisted tensor product}
\subsection{\sc The definition}\label{sectiondeftwist}
Let $C$ and $D$ be two small dg categories over $\k$. We define {\it the twisted dg tensor product} $C\sotimes D$, as follows.

The set of objects of $C\sotimes D$ is $\Ob(C)\times \Ob(D)$. Consider the graded $\k$-linear category $F(C,D)$ with objects $\Ob(C)\times \Ob(D)$ freely generated by $\{C\otimes\id_d\}_{d\in D}$, $\{\id_c\otimes D\}_{c\in C}$, and by the new morphisms $\varepsilon(f;g_1,\dots,g_n)$, specified below.

For
$$
c_0\xrightarrow{f}c_1\text{  and  }d_0\xrightarrow{g_1}d_1\xrightarrow{g_2}\dots\xrightarrow{g_n}d_n
$$
chains of composable maps in $C$ and in $D$, correspondingly, with $n\ge 1$, one introduces a morphism
$$
\varepsilon(f;g_1,\dots,g_n)\in \Hom(c_0\times d_0,c_1\times d_n)
$$
of degree
\begin{equation}
\deg \varepsilon(f;g_1,\dots,g_n)=-n+\deg f_1+\sum\deg g_j
\end{equation}

The underlying fraded category is defined as the quotient of $F(C,D)$ by the two-sided ideal, defined by the following identities:

\begin{itemize}
\item[$(R_1)$]
$(\id_c\otimes g_1)* (\id_c\otimes g_2)=\id_c\otimes (g_1g_2)$, $(f_1\otimes\id_d)*(f_2\id_d)=(f_1f_2)*\id_d$
\item[$(R_2)$] $\varepsilon(f;g_1,\dots,g_n)$ is linear in each argument,
\item[$(R_3)$] 
$\varepsilon(f; g_1,\dots,g_n)=0$ if $g_i=\id_y$ for some $y\in \Ob(D)$ and for some $1\le i\le n$,\\
$\varepsilon(\id_x; g_1,\dots,g_n)=0$ for $x\in\Ob(C)$ and $n\ge 1$,
\item[$(R_4)$]
for any $c_0\xrightarrow{f_1}c_1\xrightarrow{f_2}c_2$ and $d_0\xrightarrow{g_1}d_1\xrightarrow{g_2}\dots\xrightarrow{g_N}d_N$
one has:
\begin{equation}\label{eqsuper}
\varepsilon(f_2f_1;g_1,\dots,g_N)=\sum_{0\le m\le N}(-1)^{|f_1|(|g_{m+1}|+\dots+|g_N|+N-m)}\varepsilon(f_2;g_{m+1},\dots,g_N)\star\varepsilon(f_1;g_1,\dots,g_m)
\end{equation}
\end{itemize}
To make it a dg category, one should define the differential $d\varepsilon(f;g_1,\dots,g_n)$.

For $n=1$ we set:
\begin{equation}\label{eqd0}
\begin{aligned}
\ & -d\varepsilon (f;g)+\varepsilon(df;g)+(-1)^{|f|}\varepsilon(f;dg)=\\
& (-1)^{|f||g|}(\id_{c_1}\otimes g)\star (f\otimes \id_{d_0})
-(f\otimes \id_{d_1})\star (\id_{c_0}\otimes g)
\end{aligned}
\end{equation}
For $n\ge 2$:
\begin{equation}\label{eqd1}
\begin{aligned}
\ &\varepsilon(df;g_1,\dots,g_n)=\\
&d\varepsilon({f;g_1,\dots,g_n})-\sum_{j=1}^n(-1)^{|f|+|g_n|+\dots+|g_{j+1}|+n-j} \varepsilon({f;g_1,\dots,dg_j,\dots,g_n})\big)+(-1)^{|f|+n-1}\Big[\\
&
(-1)^{|f||g_n|+|f|}(\id_{c_1}\otimes g_n)\star \varepsilon({f;g_1,\dots,g_{n-1}})
+(-1)^{|f|+\sum_{i=2}^n(|g_i|+1)+1}\varepsilon({f;g_2,\dots,g_n})\star (\id_{c_0}\otimes g_1)+\\
&\sum_{i=1}^{n-1} (-1)^{|f|+\sum_{j=i+1}^n(|g_j|+1)  } \varepsilon({f;g_1,\dots,g_{i+1}\circ g_i,\dots,g_n})\Big]
\end{aligned}
\end{equation}

We have:
\begin{lemma}\label{lemma1}
One has $d^2=0$. The differential agrees with relations ($R_1$)-($R_4$) above. 
\end{lemma}

\qed

It is clear that the twisted tensor product $C\sotimes D$ is functorial in each argument, for dg functors $C\to C^\prime$ and $D\to D^\prime$.

Note that the twisted product $C\sotimes D$ is not symmetric in $C$ and $D$.

It is {\it not} true in general that the dg category $C\sotimes D$ is quasi-equivalent to $C\otimes D$, or that these two dg categories are isomorphic as objects of $\Hot(\Cat_\dg(\k))$. See Theorem \ref{theorht} for a result on the homotopy type of $C\sotimes D$. 

\subsection{\sc The adjunction}
Our interest in the twisted tensor product $C\sotimes D$ is explained by the following fact:
\begin{theorem}\label{prop1}
Let $C,D,E$ be three small dg categories over $\k$. 
Then there is a 3-functorial isomorphism of sets:
\begin{equation}\label{adj1}
\Phi\colon \Fun_{\dg}(C\sotimes D,E)\simeq \Fun_{\dg}(C,\Coh_\dg(D,E))
\end{equation}
\end{theorem}
\begin{proof}
Let $F\colon C\sotimes D\to E$ be a dg functor.

Define a dg functor $\Phi(F)\colon C\to\Coh_{\dg}(D,E)$, as follows:

On objects $\Phi(F)(x)=F|_{\{x\}\sotimes D}$, $x\in C$;

On morphisms: for $x_0\xrightarrow{f}x_1$ a morphism in $C$, set $\Phi(F)(f)$ to be a coherent natural transformation with components define as $F(f)$ for $n=0$, and for $n\ge 1$ its value on $y_0\xrightarrow{g_1}y_1\xrightarrow{g_2}\dots\xrightarrow{g_n}y_n$ is equal to
\begin{equation}
\Phi(F)(f)(g_1,\dots,g_n)=\prod_{n\ge 0} F(\varepsilon(f;\ g_1,\dots,g_n)\in\Hom_E(F(x_0\times y_0),F(x_1,y_n))
\end{equation}
The degree $\varepsilon(f;\ g_1,\dots,g_n))=\deg f+\sum_i\deg g_i-n$ for homogeneous $f,g_i$, and formula \eqref{eqd1} for $d(\varepsilon(f;\ g_1,\dots,g_n)$ are designed especially for $\Phi(F)(f)$ to be a coherent natural transformation. 

Finally, $\Phi(F)(f_2f_1)=\Phi(F)(f_2)\circ \Phi(F)(f_1)$ by the identity \eqref{eqsuper}. It makes $F$ a dg functor. 

\end{proof}

\begin{coroll}\label{corproj}
There is a dg functor $p_{C,D}\colon C\sotimes D\to C\otimes D$, equal to the identity on objects, and sending all $\varepsilon(f;\ g_1,\dots,g_s)$ with $s\ge 1$ to 0.
\end{coroll}
\begin{proof}
It can be either seen directly, or can be deduced from Theorem \ref{prop1} and the natural dg embedding $\FFun_\dg(D,E)\to \Coh_\dg(D,E)$, along with the classic adjunction
\begin{equation}\label{adj1bis}
\Fun_\dg(C\otimes D,E)=\Fun_\dg(C,\FFun_\dg(D,E))
\end{equation}
\end{proof}

\comment
\subsection{\sc The units}
Denote by $\k$ the dg category with a single object $e$ and with $\Hom(e,e)=\k$.

For a small dg category $C$, one has $C\sotimes \k\ne C$ and $\k\sotimes C\ne C$, as it can be seen from our definition. 

Nevertheless one has:
\begin{prop}
Let $C$ be a small dg category over $\k$. There are the following unit dg functors:
\begin{equation}
\begin{aligned}
\ &\lambda_C \colon \k\sotimes C\to C\\
&\lambda^\prime_C\colon C\sotimes \k\to C\\
&\rho_C\colon C\to C\sotimes \k
\end{aligned}
\end{equation}
\end{prop}
\begin{proof}
The dg functor $\lambda_C$ (resp., $\lambda_C^\prime$) is defined as the compositions of $p_{\k,C}$ of Corollary \ref{corproj} followed by the isomorphism $\k\otimes C=C$ (resp., the composition of $p_{C,\k}$ followed by the isomorphism $C\sotimes \k=C$).

For $\rho_C$ it is enough to define, in view of adjunctions \eqref{adj1}, \eqref{adj1bis}, a dg functor $P\colon \Coh_{\dg}(\k,C)\to 
\FFun_\dg(\k,C)$. It is defined as follows.

For both categories $\FFun_\dg(\k,C)$, $\Coh_\dg(\k,C)$, the objects are dg functors $\k\to C$. A dg functor $\k\to C$ is identified with an object $x$ of $C$. For two objects $x,y\in C$, the $\Hom$-complex $\FFun_\dg(\k,C)(x,y)$ is identified with the complex $\Hom_C(x,y)=V$. The $\Hom$-complex $\Coh_\dg(\k,C)(x,y)$ is the (product) total complex of the following bicomplex $W$:
$$
\underset{\deg=0}{V}\xrightarrow{0}\underset{\deg=1}{V}\xrightarrow{\id}\underset{\deg=2}{V}\xrightarrow{0}\underset{\deg=3}{V}\xrightarrow{\id}\dots
$$
There is a map of bicomplexes $W\to V$ sending the components $V$ in degrees 1,2, ... to 0. It induces a map of total complexes.

\end{proof}

Below we define the associativity constraint $\alpha_{C,D,E}\colon (C\sotimes D)\sotimes E\to C\sotimes (D\sotimes E)$ which is not, in general, an isomorphism. We show that the data $(\sotimes, \alpha_{C,D,E}, \lambda_C,\rho_C)$ fulfils the axioms of a {\it skew monoidal category} [LS].

Note that a dg functor $\rho^\prime_C\colon C\to \k\sotimes C$ in general does not exist. 

\endcomment

\subsection{\sc The homotopy type of $C\sotimes D$}
For general $C,D$, we do not know the homotopy type of the dg category $C\sotimes D$. However, one has:
\begin{theorem}\label{theorht}
Let $C,D$ be small dg categories over $\k$. Assume both $C,D$ are cofibrant for the Tabuada closed model structure.
Then $C\sotimes D$ is also cofibrant and is isomorphic to $C\otimes D$ as an object of $\Hot(\Cat_\dg(\k))$.
\end{theorem}

We prove Theorem \ref{theorht} in Section \ref{proofht} below.

The main step of the proof is to show that, for cofibrant $C$ and $D$, the adjunction \eqref{adj1} ``descends'' to the adjunction in the homotopy category of dg categories $\Hot$:
\begin{equation}\label{adj1hot}
\Hot(C\sotimes D,E)=\Hot(C,\Coh_\dg(D,E))
\end{equation}
As soon as \eqref{adj1hot} is established, Theorem \ref{theorht} follows from To\"{e}n's result stated in Theorem \ref{theortoen}, from Lemma \ref{lemmastep0}, Theorem \ref{theoremfa}, and from the Yoneda lemma.

Thus, the main step is to pass from \eqref{adj1} to \eqref{adj1hot}.

The rest of the paper is devoted to a proof of Theorem \ref{theorht}.

\comment
The proof is based on several Lemmas.
\begin{lemma}\label{lemmafin1}
Let $C,D$ be cofibrant. Then $C\sotimes D$ is also cofibrant.
\end{lemma}

???

\qed 

Denote by $\k[t,dt]$ the skew-commutative dg algebra over $\k$ with generators $t$ in degree 0 and $dt$ in degree 1, and with the differential sending $t$ to $dt$. For any $a\in \k$ denote by $p_a\colon \k[t,dt]\to \k$ the dg algebra map, sending $t$ to $a$ and $dt$ to 0.

\begin{lemma}\label{lemmafin2}
Let $C$ be cofibrant. Then $\Hot(C,D)$ can be computed as the quotient set $\Fun_\dg(C,D)/\sim$, where two dg functors $F,G\colon C\to D$ are equivalent if there is a dg functor $H\colon C\to D\otimes \k[t,dt]$ such that $p_0\circ H=F$, $p_1\circ H=G$. 
\end{lemma}

???

\qed

{\it Proof of Theorem:}

We compute $\Hot(C\sotimes D,E)$ and $\Hot(C,\Coh_\dg(D,E))$ by Lemma \ref{lemmafin2}, as
\begin{equation}\label{eqam1}
\Hot(C\sotimes D,E)=\Fun_\dg(C\sotimes D,E)/\sim,\ \ \ \Hot(C,\Coh_\dg(D,E))=\Fun_\dg(C,\Coh_\dg(D,E))/\sim
\end{equation}
The first equation follows from Lemma \ref{lemmafin2}, because $C\sotimes D$ is cofibrant, by Lemma \ref{lemmafin1}.

The dg functors in the right-hand sides are identified, by Theorem \ref{prop1}. We need to ``upgrade'' this identification to the homotopy relations.

Lemma \ref{lemmafin2} states that $F\sim G$ if there are $H_1, H_2$ which belong to
\begin{equation}
H_1\in \Fun_\dg(C\sotimes D,E\otimes \k[t,dt]),\ \ H_2\in\Fun_\dg(C,\Coh_\dg(D,E)\otimes \k[t,dt])
\end{equation}
such that $p_0\circ H_i=F$, $p_1\circ H_i=G$, $i=1,2$.

One has:
$$
\Fun_\dg(C\sotimes D,E\otimes \k[t,dt])=\Fun_\dg(C,\Coh_\dg(D,E\otimes \k[t,dt]))
$$
by Theorem \ref{prop1}.

We need to find a correspondence between
\begin{equation}
\Fun_\dg(C,\Coh_\dg(D,E\otimes \k[t,dt]))\text{  and  }\Fun_\dg(C,\Coh_\dg(D,E)\otimes \k[t,dt])
\end{equation}
which agrees with the projections $p_0, p_1$ to $\Fun_\dg(C,\Coh_\dg(D,E))$.

We claim that there are maps of sets $\alpha,\beta$, making the diagram 
\begin{equation}
\xymatrix{
\Fun_\dg(C,\Coh_\dg(D,E\otimes \k[t,dt]))\ar[rr]<1ex>^{\alpha}\ar[ddr]_{p_a}&&\Fun_\dg(C,\Coh_\dg(D,E)\otimes \k[t,dt])\ar[ll]<1ex>^{\beta}\ar[ddl]^{p_a}\\
\\
&\Fun_\dg(C,\Coh_\dg(D,E))
}
\end{equation}
commutative, for $a=0,1$.

We construct dg functors $A,B$ making the diagram commutative, for $a=0,1$:
\begin{equation}
\xymatrix{
\Coh_\dg(D,E\otimes \k[t,dt])\ar[rr]<1ex>^{A}\ar[ddr]_{p_a}&&\Coh_\dg(D,E)\otimes \k[t,dt]\ar[ll]<1ex>^{B}\ar[ddl]^{p_a}\\
\\
&\Coh_\dg(D,E)
}
\end{equation}

\endcomment

\section{\sc Reminder on closed model categories}
\subsection{\sc The Homotopy relation}\label{sectionlrhot}
Let $\mathscr{M}$ be a closed model category [Q], [GJ, Ch. II, Section 1], [Hir, Ch. 7]. Recall that one derives a {\it homotopy relation} on $\Hom_\mathscr{M}(C,D)$ from the basic axioms of a closed model category. There are two such relations. 

The first one $\sim_L$ uses {\it a cylinder object} of $C$, and is an equivalence relation when $C$ is cofibrant, in which case the realation $f\sim_L g$ does not depend on the choice of a cylinder object. 

The second one $\sim_R$ uses {\it a path object} of $D$, and is an equivalence relation when $D$ is fibrant, in which case the relation $f\sim_R g$ does not depend on the choice of a path object.

Moreover, when $C$ is cofibrant and $D$ is fibrant, $f\sim_L g$ holds iff $f\sim_R g$, see Proposition \ref{proplr} below.

{\it A cylinder object} of an object $C$ of $\mathscr{M}$ is a commutative triangle 
\begin{equation}
\xymatrix{
C\sqcup C\ar[dr]^{\nabla}\ar[d]_{i}\\
\tilde{C}\ar[r]_{\sigma}&C
}
\end{equation}
where $\nabla\colon C\sqcup C\to C$ is the canonical map defined from the identity map $C\to C$ on each summand, $i$ is a cofibration, and $\sigma$ is a weak equivalence. 

The cylinder object always exists but is not unique. 

The {\it left homotopy relation} $f\sim_L g$, for $f,g\colon C\to D$, is defined for a given choice of a cylinder object of $C$. It is defined as a commutative diagram
\begin{equation}
\xymatrix{
C\sqcup C\ar[d]_{i}\ar[dr]^{(f,g)}\\
\tilde{C}\ar[r]_h&D
}
\end{equation}
where $(f,g)$ is the map defined via $f$ and $g$ on the summands $C$, and the data $C\sqcup C\xrightarrow{i}\tilde{C}$ comes from some choice of a cylinder object.

In general, $\sim_L$ is not an equivalence relation on $\Hom(C,D)$, but it is when $C$ is cofibrant. See Proposition \ref{proplr} below. 

Recall also the path objects and the right homotopy relation.

Let $\mathscr{M}$ be a closed model category, $D$ an object of $\mathscr{M}$.

A {\it path object} of $D$ is a commutative triangle
\begin{equation}\label{eqpath}
\xymatrix{
&\hat{D}\ar[d]^{p=(p_0,p_1)}\\
D\ar[ur]^s\ar[r]_{\Delta\hspace{12pt}}&D\times D
}
\end{equation}
where $\Delta$ is the diagonal map, $s$ is a weak equivalence, and $p$ is a fibration.  The path objects always exist but are not unique.

Assume a path object of $D$ is chosen. One says that $f,g\in\Hom(C,D)$ are {\it right homotopic} with respect to the path object if there is a commutative triangle
\begin{equation}\label{equationrhot}
\xymatrix{
&\hat{D}\ar[d]^{p=(p_0,p_1)}\\
C\ar[ur]^{\theta}\ar[r]_{(f,g)\hspace{12pt}}&D\times D
}
\end{equation}

In general, $\sim_R$ is not an equivalence relation on $\Hom(C,D)$, but it is, if $D$ is fibrant, see Proposition \ref{proplr} below. 

\begin{prop}\label{proplr}
Let $\mathscr{M}$ be a closed model category. The following statements are true:
\begin{itemize}
\item[(i)] For a cofibrant object $C$, the left homotomy relation $\sim_L$ on $\Hom(C,D)$ is an equivalence relation. Moreover, in this case, if $f\sim_L g$ for one choice of the cylinder object of $C$, then $f\sim_L g$ for any other choice,
\item[(ii)] For a fibrant object $D$, the right homotopy relation $\sim_R$ on $\Hom(C,D)$ is an equivalence relation. Moreover, in this case, if $f\sim_R g$ for one choice of the path object of $D$, then $f\sim_R g$ for any other choice,
\item[(iii)] Let $C$ be cofibrant and $D$ be fibrant. Then on $\Hom(C,D)$ both homotopy relations coincide: $f\sim_L g\Leftrightarrow f\sim_R g$.
\end{itemize}
\end{prop}
See [GJ, Cor. II 1.9] for a proof.

\qed

The {\it homotopy category} $\Ho(\mathscr{M})$ is defined as a category whose objects are the same as the objects of $\mathscr{M}$, and whose morphisms $\Ho(\mathscr{M})(C,D)$ is $\mathscr{M}(RQ(C),RQ(D))/\sim_L$, or equally, $\mathscr{M}(RQ(C),RQ(D))/\sim_R$, where $RQ(X)$ is a fibrant cofibrant replacement of the object $X$.
The reader is refered to [GJ, Ch. II, Section 1] for more detail. 

\subsection{\sc The closed model structure on the category of dg associoative algebras}
Recall that a dg associative algebra $A$ over $\k$ is called {\it semi-free} if there is an exhaustive ascending filtration $$A_0\subset A_1\subset A_2\subset \dots$$
where each $A_i$ is a dg associative algebra, the underlying graded algebra of $A_{i+1}$ is freely generated by $A_i$ and a set of elements $V_{i+1}=\{f_{i+1,1},\dots,f_{i+1,n_{i+1}}\}$, and 
$$
d(f_{i+1,k})\subset A_i \text{  for any $k$}
$$
(the latter codition for $i=-1$ means that $d(f_{0,k})=0$ for any $k$).

Similarly, let in the above definition $A_0$ be some dg associative algebra (with possibly non-zero differential), and one has an exhaustive ascending filatration $A_0\subset A_1\subset A_2\subset\dots$ on $A$, such that each $A_i$ is a dg associative algebra, the underlying graded algebra of $A_{i+1}$ is freely generated by $A_i$ and a set of elements $V_{i+1}=\{f_{i+1,1},\dots,f_{i+1,n_{i+1}}\}$, and 
$$
d(f_{i+1,k})\subset A_i \text{  for any $k$ and for $i\ge 0$}
$$
(that is, the condition for $i=-1$ is dropped). Then the imbedding $$i\colon A_0\to A$$ is called {\it a standard cofibration}.

Recall that the category of dg associative algebras over a field $\k$ admits a closed model structure due to Hinich [H1]. Its weak equivalences are quasi-isomorphisms, fibrations are component-wise surjective maps, and cofibrations are the maps satisfying the left lifting property with respect to acyclic fibrations, see [H1, Theorem 2.2.1]. One can prove that a map in this closed model structure is a cofibration if it is a retract of a standard cofibration, see [H1, Remark 2.2.5].

\subsection{\sc The closed model structure on the category of small dg categories}\label{sectioncmcdg}
Introduce some notations, cf. [Tab1, Sect. 2]. Denote by $D^n$ the complex
$$
\dots 0\to \underset{\deg=-n}{\k}\xrightarrow{\id}\underset{\deg=-n+1}{\k}\to 0\dots
$$
It can also be defined as the cone of the identity map of $\k[n-1]$.

Define two dg categories, $C(n)$ and $P(n)$, $n\in\mathbb{Z}$.

The dg category $C(n)$ has 2 objects, denoted by $a,b$,  and the morphisms $C(n)(a,a)=\k$, $C(n)(b,b)=\k$, $C(n)(b,a)=0$, $C(n)(a,b)=\k[n-1]$. 

The dg category $P(n)$ has 2 objects, denoted by $a^\prime,b^\prime$, and the morphisms $P(n)(a^\prime,a^\prime)=\k$, $P(n)(b^\prime,b^\prime)=\k$, $P(n)(b^\prime,a^\prime)=0$, $P(n)(a^\prime,b^\prime)=D^n$. 

There is a dg functor $s(n)\colon C(n)\to P(n)$. It sends $a$ to $a^\prime$, $b$ to $b^\prime$, and the corresponding map $C(n)(a,b)\to P(n)(a^\prime,b^\prime)$ is the map $s_n\colon \k[n-1]\to D^n$ which maps $\k[n-1]$ to the summand $\k[n-1]\subset D^n$, by the identity map.

Let $X$ be a dg category. Note that a dg functor $F\colon C(n)\to X$ is nothing but a pair of objects $x,y\in X$, and a closed homogeneous of degree $-r+1$ morphism $X(x,y)$. The dg functor $F$ factors as $C(n)\xrightarrow{s(n)}P(n)\to X$ if and only if the closed homogeneous morphism of degree $-r+1$ is a coboundary.

A dg functor $f\colon X\to Y$ is called {\it a relative $I$-cell complex} if there is an ascending filtration $X=Y_0\subset Y_1\subset Y_2\subset \dots$, and maps $f_n\colon X\to Y_n$  such that $j_n\circ f_n=f_{n+1}$ (where $j_n\colon Y_n\to Y_{n+1}$ are the embedding of the consequtive terms of the filtration), such that $f_0\colon X\to Y_0=X$ is the identity functor, and $f_{n+1}\colon X\to Y_{n+1}$ is obtained from $f_n\colon X\to Y_n$ by either of the following operations (i), (ii):
\begin{itemize}
\item[(i)] $Y_{n+1}=Y_n\sqcup \underline{\k}$ (here $\underline{\k}$ is a dg category with a single object whose complex of endomorphisms is $\k$), and $f_{n+1}$ is the composition $X\xrightarrow{f_n}Y_n\xrightarrow{\alpha}Y_n\sqcup\underline{\k}$, where $\alpha$ is the canonical map to the first summand;
\item[(ii)] we are given a dg functor $F\colon C(k)\to Y_n$ (for some $k\in\mathbb{Z}$), and define $Y_{n+1}$ as colimit of the pushout diagram:
\begin{equation}
\xymatrix{
C(k)\ar[r]^{F}\ar[d]_{s(k)}&Y_{n}\ar@{.>}[d]^{j_n}\\
P(k)\ar@{.>}[r]_{F^\prime}&Y_{n+1}
}
\end{equation}
Define $f_{n+1}\colon X_n\to Y_{n+1}$ as $f_{n+1}=j_n\circ f_n$.
\end{itemize}

A small dg category $C$ is called {\it an $I$-cell complex} if $\varnothing\to C$ is a relative $I$-cell complex. 

\begin{remark}{\rm
Strictly speaking, one needs to consider the filtrations labelled by all small ordinals $\lambda$ in the definition of a relative $I$-cell complex, and to use the transfinite induction, cf. [Hir, Ch. 10.1-10.5]. Although the small object argument in our case holds for all ordinals starting with $\aleph_0$, one needs to have more relative $I$-cell complexes and $I$-cell complexes for e.g. Proposition \ref{propretract} below to be true. We skip these technical issues addressing the interested reader to loc.cit.
}
\end{remark}

Recall that the category $\Cat_\dg(\k)$ of small dg categories over $\k$ admits a cofibrantly generated closed model structure, constructed in [Tab1].

Its weak equivalences are quasi-equivalences of dg categories.

Its set of generating cofibrations is $I=\{\alpha, s(n)\}_{n\in\mathbb{Z}}$, where $\alpha\colon \varnothing\to\underline{\k}$ is the unique dg functor from the initial object $\varnothing$ in $\Cat_\dg(\k)$ to $\underline{\k}$.

The fibrations are described as follows:

A dg functor $F\colon C\to D$ is called a {\it fibration} if the following conditions hold:
\begin{itemize}
\item[(F1)] for any two objects $x,y\in C$, the map $C(x,y)\to D(F(x),F(y))$ induced by $F$, is a fibration of complexes (that is, a component-wise surjection), 
\item[(F2)] for any $x\in C$ and any isomorphism $v\colon [F](x)\to y^\prime$ in $H^0(D)$ there exists {\it an isomorphism} $u\colon x\to y$ in $H^0(C)$ such that $[F](u)=v$.
\end{itemize}

In particular, any object is fibrant.

The cofibrations are defined as those morphisms which have the left lifting property with respect to the acyclic fibrations. 

One has the following explicit description of the cofibrations.
\begin{prop}\label{propretract}
Any cofibration for the closed model structure on $\Cat_\dg(\k)$ is a retract of a relative $I$-cell complex. 
\end{prop}
\begin{proof}
It follows from the general description of cofibrations in a cofibrantly generated closed model category, given in [Hir, Prop. 11.2.1 (1)], and from the description of generating cofibrations, see [Tab1, Th. 1.8] for more detail.
\end{proof}

\subsection{\sc Explicit path objects}
\subsubsection{\sc The case of dg associative algebras}
Assume $\cchar \k=0$. Consider the dg commutative algebra (the algebraic de Rham complex of $\mathbb{A}_\k^1$), equal to $\k[t,dt]$, $d(t)=dt, d(dt)=0$. It is well-known that $\k[t,dt]$ is quasi-isomorphic to $\k[0]$. For a dg associative $C$, set
$$
\hat{C}=C\otimes\k[t,dt]
$$
The two projections $p_0,p_1\colon C\otimes \k[t,dt]\to C$ are given by two maps of dg algebras $p_0^\prime,p_1^\prime\colon \k[t,dt]\to \k$, which are evaluations at $t=0$ and at $t=1$. The map $(p_0,p_1)\colon C\otimes \k[t,dt]\to C\times C$ is term-wise surjective.

The map $s\colon C\to C\otimes \k[t,dt]$ is given by the dg algebra map $\k\to\k[t,dt]$.

It is clear that the diagram 
$$
\xymatrix{
&C\otimes \k[t,dt]\ar[d]^{(p_0,p_1)}\\
C\ar[ru]^{s}\ar[r]_{\Delta}&C\times C
}
$$
is a path object.

This simple construction is not generalized directly for the case of dg categories. 

\subsubsection{\sc The case of small dg categories}\label{sectionpathtab}
Here we provide a detailed account on the Tabuada path object $\hat{C}$ for a small dg category $C$, see [Tab2, Sect. 2]. For our needs, we provide a direct proof that it is a path object, replacing the implicit part in the proof of [Tab2, Prop. 2.0.11] by an explicit argument. It will make us possible to prove that for small dg categories $C,D$, with $C$ an $I$-cell complex, both dg categories $\widehat{\Coh_\dg(C,D)}$ and $\Coh_\dg(C,\hat{D})$, provide path objects of the dg category $\Coh_\dg(C,D)$.

Let $C$ be a small category. The category $\hat{C}$ has as objects the triples $(x,y,f)$, where $x,y\in \Ob(C)$, $\phi$ a closed degree 0 morphism $f\colon x\to y$, such that the corresponding morphism $[f]\in H^0(C)(x,y)$ is an isomorphism. 

The complexes of morphisms defined (as $\mathbb{Z}$-graded modules) as
$$
\hat{C}(x\xrightarrow{f}y, w\xrightarrow{g}z)=C(x,w)\oplus C(y,z)\oplus C(x,z)[-1]
$$
A homogeneous morphism $\phi$ of degree $r$ is given by a matrix
$$
\phi=\begin{bmatrix}m_1&0\\h&m_2\end{bmatrix}
$$
and its differential is given by
\begin{equation}
d\phi=\begin{bmatrix}
dm_1&0\\
dh+g\circ m_1- (-1)^r m_2\circ f& dm_2
\end{bmatrix}
\end{equation}
where $m_1$ and $m_2$ are homogeneous of degree $r$, and $h$ is homogeneous of degree $r-1$.

For $\phi$ as above and $$\phi^\prime=\begin{bmatrix}m_1^\prime&0\\h^\prime&m_2^\prime\end{bmatrix}\in\hat{C}(w\xrightarrow{g}z, p\xrightarrow{t}q)$$
the composition $\phi^\prime\circ \phi$ is defined as
\begin{equation}
\phi^\prime\circ\phi=\begin{bmatrix}
m_1^\prime m_1&0\\
h^\prime m_1+m_2^\prime h& m_2^\prime m_2
\end{bmatrix}\in\hat{C}(x\xrightarrow{f}y,p\xrightarrow{t}q)
\end{equation}
To form diagram \eqref{eqpath}, one needs to define dg functors $s\colon C\to \hat{C}$ and projections $p_0,p_1\colon \hat{C}\to C$. 
The dg functor $s$ sends $x\in C$ to $x\xrightarrow{\id}x$ in $\hat{C}$, and sends a morphism $m\in C(x,w)$ to the morphism $\hat{C}(x\xrightarrow{\id}x,w\xrightarrow{\id}w)$ given by the
matrix 
$$
\begin{bmatrix}
m&0\\
0&m
\end{bmatrix}
$$
The projection $p_0$ (corresp., $p_1$) sends $x\xrightarrow{f}y$ to $x$ (corresp., to $y$), and is extended to morphisms in the natural way. The dg functor $(p_0,p_1)$ sends the morphism $\phi\in\hat{C}(x\xrightarrow{f}y,w\xrightarrow{g}z)$ as above to the morphism in $(C\times C)(x\times y, w\times z)$, given as $(m_1\times m_2)$.

One has:
\begin{prop}\label{proptabpath}
For a small dg category $C$, the dg category $\hat{C}$ embedded to the diagram
\begin{equation}
\xymatrix{
&\hat{C}\ar[d]^{(p_0,p_1)}\\
C\ar[ur]^{s}\ar[r]_{\Delta}&C\times C
}
\end{equation}
is a path object. That is, the dg functor $s$ is a weak equivalence, and the dg functor $(p_0,p_1)$ is fibration. 
\end{prop}
\begin{proof}
The statement that $s$ is a weak equivalence is standard, and the reader is referred to [Tab2, Prop. 2.0.11] for a proof. The condition (F1) of a fibration (see Section \ref{sectioncmcdg}) is tautological. The hard part is the property (F2). As we've said, we provide a direct proof of this statement, different from loc.cit. It is a bit computational. We use this direct approach in Lemma \ref{lemmax} below, in the proof of Theorem \ref{theorht}.

In this setting, (F2) is the following statement.
\begin{lemma}\label{lemmaf2}
Let $x\xrightarrow{f}y$ be an object of $\hat{C}$, and let $x\xrightarrow{a}w$ and $y\xrightarrow{b}z$ be closed degree 0 morphisms which are isomorphisms in $H^0(C)$. Then the solid arrow diagram below 
\begin{equation}\label{f2lemma2}
\xymatrix{
x\ar[r]^a\ar[d]_{f}\ar@{.>}[dr]^{\xi}&w\ar@{.>}[d]^{g}\\
y\ar[r]_{b}&z
}
\end{equation}
can be reconstructed by the dashed arrows to a closed morphism 
$\Phi\in \hat{C}(x\xrightarrow{f}y, w\xrightarrow{g}z)$, that is
\begin{equation}\label{f2lemma3}
d\xi=b\circ f- g\circ a
\end{equation}
Moreover, $\Phi$ becomes an isomorphism in $H^0(\hat{C})$.

\end{lemma}
\begin{proof}

Recall the following classic fact, due to M.Kontsevich [Ko, Lecture 6]:
\begin{lemma}\label{lemmamax}
Let $C$ be a dg category, and let $f\colon X\to Y$ be a closed degree 0 morphism, such that $[f]$ is an isomorphism in $H^0(C)$. 
Then there exist the following data: a closed degree 0 morphism $g\colon Y\to X$, morphisms $h_X\in \Hom^{-1}_C(X,X)$, $h_Y\in\Hom^{-1}_C(Y,Y)$, and a morphism $r\in \Hom^{-2}_C(X,Y)$, such that
\begin{equation}\label{max1}
gf=\id_X+dh_X,\ \ fg=\id_Y+dh_Y
\end{equation}
\begin{equation}\label{max2}
fh_X-h_Yf=dr
\end{equation}
\end{lemma}
Of course, \eqref{max1} is trivial. The statement of Lemma essentially means that it is always possible to choose $g,h_X,h_Y$ such that \eqref{max2} holds.
\begin{proof}
Choose $g,h_X,h_Y$ in an arbitrary way. Then replace $h_Y$ by
\begin{equation}
h^\prime_Y=h_Y+fh_Xg-h_Yfg
\end{equation}
(and keep $g$ and $h_X$ unchanged).

It is claimed that $(f,g,h_X,h_Y^\prime)$ satisfy both \eqref{max1} and \eqref{max2}. 
The equation $fg=\id_Y+dh^\prime_Y$ is checked directly. For the equation \eqref{max2}, one has:
$$
fh_X-h^\prime_Yf=fh_X+h_Yfgf-h_Yf-fh_Xgf=d(-h_Yfh_X+fh_Xh_X)
$$
\end{proof}
\begin{remark}{\rm
V.Drinfeld constructed in [Dr, 3.7] a semi-free dg category with 2 objects $a,b$, which is a resolution of the $\k$-linear envelope of the ordinary category with two objects $a,b$, having exactly 1 morphism between any two objects. The construction was inspired by the Lemma above. The Drinfeld dg category has a fundamental value in Tabuada's construction [Tab1] of closed model structure on the category of small dg categories.
}
\end{remark}
We continue proving Lemma \ref{lemmaf2}.

Find closed degree 0 maps $a^\prime\colon w\to x$ and $b^\prime \colon z\to y$ which are inverse in $H^0(C)$:
\begin{equation}\label{f2lemma4}
\begin{aligned}
\ &aa^\prime=\id_w+dh_w\\
&a^\prime a=\id_x+dh_x\\
&bb^\prime=\id_z+dh_z\\
&b^\prime b=\id_y+dh_y
\end{aligned}
\end{equation}
and
\begin{equation}\label{f2lemma4bis}
\begin{aligned}
\ &a^\prime h_w-h_xa^\prime=dr_1\\
&bh_y-h_zb=dr_2
\end{aligned}
\end{equation}
(It is possible by Lemma \ref{lemmamax}).

Now define $g$ and $\xi$ as
\begin{equation}\label{f2lemma5}
\begin{aligned}
\ &g:=bfa^\prime\\
&\xi:=-bfh_x
\end{aligned}
\end{equation}
One checks that \eqref{f2lemma3} holds, that is, we have constructed a morphism $\Phi\in \hat{C}(x\xrightarrow{f}y,w\xrightarrow{g}z)$ such that $(p_0,p_1)(\Phi)=(a,b)$. It proves the first assertion. 

It remains to show that $\Phi$ becomes an isomorphism in $H^0(\hat{C})$. 

Define a morphism $\Phi^\prime\in \hat{C}(w\xrightarrow{g}z, x\xrightarrow{f}y)$ as
\begin{equation}\label{f2lemma6}
\xymatrix{
w\ar[r]^{a^\prime}\ar[d]_{g}\ar@{.>}[rd]^{\mu}&x\ar[d]^{f}\\
z\ar[r]_{b^\prime}&y
}
\end{equation}
with
\begin{equation}\label{f2lemma7}
\mu=h_yfa^\prime
\end{equation}
One has
\begin{equation}\label{f2lemma8}
d\mu=b^\prime g-fa^\prime
\end{equation}
that is, $\Phi^\prime$ is closed in $\hat{C}$.

It remains to show that
\begin{equation}\label{f2lemma9}
\begin{aligned}
\ &\Phi^\prime \Phi=\id({x\xrightarrow{f}y})+d\Gamma\\
&\Phi \Phi^\prime=\id({w\xrightarrow{g}z})+d\Gamma^\prime
\end{aligned}
\end{equation}
for some morphisms $\Gamma\in \hat{C}(x\xrightarrow{f}y,x\xrightarrow{f}y),\Gamma^\prime\in\hat{C}(w\xrightarrow{g}z,w\xrightarrow{g}z)$.

The composition $\Phi^\prime\Phi$ is
\begin{equation}
\xymatrix{
x\ar[rr]^{a^\prime a}\ar[d]_{f}\ar[rrd]^{s}&&x\ar[d]^{f}\\
y\ar[rr]_{b^\prime b}&&y
}
\end{equation}
where
\begin{equation}
s=-b^\prime bfh_x+h_yfa^\prime a=(-fh_x+h_yf)+d(-h_yfh_x)
\end{equation}

Define $\Gamma$ as 
\begin{equation}\label{f2lemma10}
\xymatrix{
x\ar[r]^{h_x}\ar[d]_{f}\ar[dr]^{\epsilon}&x\ar[d]^{f}\\
y\ar[r]_{h_y}&y
}
\end{equation}
where
\begin{equation}\label{f2lemma11}
\epsilon=-h_yfh_x
\end{equation}

Quite surprisingly, the case of composition $\Phi\Phi^\prime$ is more tricky. One needs Lemma \ref{lemmamax} for this case.

The composition $\Phi\Phi^\prime$ is equal to
\begin{equation}
\xymatrix{
w\ar[rr]^{aa^\prime}\ar[d]_{g}\ar[rrd]^t&&w\ar[d]^{g}\\
z\ar[rr]_{bb^\prime}&&z
}
\end{equation}
where 
\begin{equation}
t=bh_yfa^\prime-bfh_xa^\prime=(h_zb+dr_2)fa^\prime-bf(a^\prime h_w-dr_1)=(h_zg-g h_w)  +d(r_2 fa^\prime-bfr_1)
\end{equation}
where $r_1,r_2$ are defined in \eqref{f2lemma4bis}.

It follows that $\Phi\Phi^\prime=\id(w\xrightarrow{g}z)+d\Gamma^\prime$, where $\Gamma^\prime$ is equal to 
\begin{equation}
\xymatrix{
w\ar[r]^{h_w}\ar[d]_{g}\ar[rd]^{\epsilon^\prime}&w\ar[d]^{g}\\
z\ar[r]_{h_z}&z
}
\end{equation}
where
\begin{equation}\label{eqstrange}
\epsilon^\prime=r_2 fa^\prime-bfr_1
\end{equation}

\end{proof}

\end{proof}

\section{\sc A Proof of Theorem \ref{theorht}}\label{proofht}
\subsection{\sc The idea}
Let $C$ and $D$ be cofibrant. We need to show that the isomorphism
\begin{equation}\label{again1}
\Fun_\dg(C\sotimes D,E)\simeq \Fun_\dg(C,\Coh_\dg(D,E))
\end{equation}
descends to a map in the homotopy category (and that this map is an isomorphism):
\begin{equation}\label{again2}
\Hot(C\sotimes D,E)\simeq \Hot(C,\Coh_\dg(D,E))
\end{equation}
For $X$ cofibrant, we derive the equivalence relation on $\Fun_\dg(X,Y)$ via a path object $\hat{Y}$ of $Y$, as in \eqref{equationrhot}. That is, we consider the {\it right} homotopy relation $\sim_R$. 

For a small dg category $Y$, denote by $\hat{Y}$ the Tabuada path object of $Y$, see Section \ref{sectionpathtab}.
It is a part of the commutative diagram
$$
\xymatrix{&\hat{Y}\ar[d]^{(p_0,p_1)}\\
Y\ar[ur]^{s}\ar[r]_{\Delta}&Y\times Y}
$$
where $s$ is a weak equivalence and $(p_0,p_1)$ is a fibration.

One has from \eqref{again1}:
\begin{equation}\label{again3}
\Fun_{dg}(C\sotimes D,\hat{E})\simeq \Fun_\dg(C,\Coh_\dg(D,\hat{E}))
\end{equation}
To derive \eqref{again2} from \eqref{again3}, one needs to show
\begin{klemma}\label{keylemma}
Let $C,D$ be small dg categories, with $C$ an $I$-cell complex. Then $\Coh_\dg(C,\hat{D})$ is a path object of $\Coh_\dg(C,D)$.
That is, in the natural diagram
\begin{equation}
\xymatrix{&\Coh_\dg(C,\hat{D})\ar[d]^{(p_{0*},p_{1*})}\\
\Coh_\dg(C,D)\ar[ur]^{s_*}\ar[r]_{\Delta\hspace{12mm}}&\Coh_\dg(C,D)\times\Coh_\dg(C,D)
}
\end{equation}
the map $s_*$ is a weak equivalence and the map $(p_{0*},p_{1*})$ is fibration.

Here the maps $p_{0*},p_{1*}\colon \Coh_\dg(C,\hat{D})$ are induced by the maps $p_0,p_1\colon \hat{D}\to D$, and the map $s_*$ is induced by the map  $s\colon D\to \hat{D}$.
\end{klemma}

The statement that $s_*$ is a weak equivalence follows from Corollary \ref{corollstep0}(ii) below, because $s\colon D\to \hat{D}$ is a weak equivalence by Proposition \ref{proptabpath}.

Furthermore, the axiom (F1) of fibrations (see Section \ref{sectioncmcdg}) holds for $(p_{0*},p_{1*})$ by elementary reasons.

The hardest part is to prove the axiom (F2).  It is given in Section \ref{sectionkl} below. In Section \ref{klimpliesth} we prove Theorem \ref{theorht}, assuming Key-Lemma \ref{keylemma}.

\subsection{\sc Theorem \ref{theorht} follows from Key-Lemma \ref{keylemma}}\label{klimpliesth}
Here we deduce Theorem \ref{theorht} from Key-Lemma \ref{keylemma}. The deduction uses quite standard arguments. It is diveded into several steps.

\vspace{3mm}

{\it Step 0.}
\begin{lemma}\label{lemmastep0}
Let $C,D$ be small dg categories, with $C$ cofibrant. Then the natural imbedding 
$$
i\colon \Coh_\dg(C,D)\to\Coh_{A_\infty}(C,D)
$$
is a weak equivalence.
\end{lemma}
\begin{proof}
An $A_\infty$ functor $C\to D$ is a dg functor $\Cobar(\Bar(C))\to D$. There is a canonical projection $p\colon \Cobar(\Bar(C))\to C$ which is a weak equivalence. Consider the commutative diagram of dg functors
\begin{equation}
\xymatrix{
\Coh_\dg(C,D) \ar[rr]^i\ar[rd]_{t}&&   \Coh_{A_\infty}(C,D)\ar[dl]^s\\
&\Coh_\dg(\Cobar(\Bar(C)),D)
}
\end{equation}
The dg functor $s$ is an isomorphism on sets of objects, and a quasi-isomorphism on the corresponding complexes of maps by a standard argument (it holds for any small dg category $C$, not necessarily cofibrant).

The dg functor $t=p^*$ is defined on objects as the pre-composition with $p$, and is defined on morphisms as
\begin{equation}
t(\Psi)(f_1\otimes\dots\otimes f_n)=\Psi(pf_1\otimes pf_2\otimes \dots pf_n)
\end{equation}
for $\Psi\in\Coh_\dg(\Cobar(\Bar(C)),D)(tF,tG)$,
where $F,G\colon C\to D$ are dg functors, $tF(-)=F(p(-)), tG(-)=G(p(-))$.

For any $F,G$ as above, the maps $t\colon\Coh_\dg(C,D)(F,G)\to\Coh_\dg(\Cobar(\Bar(C)),D)(tF,tG)$ is a quasi-isomorphism of complexes, because $p$ is a weak equivalence (here we do not use that $C$ is cofibrant).

The point where the cofibrancy of $C$ is used is to show that $H^0(t)\colon H^0(\Coh_\dg(C,D))\to H^0(\Cobar(\Bar(C)),D))$ is equivalence of $\k$-linear categories.

Let us prove that $H^0(t)$ is an equivalence of $\k$-linear categories.

In the diagam 
\begin{equation}
\xymatrix{
\emptyset \ar[d]\ar[r]&\Cobar(\Bar(C))\ar[d]^{p}\\
C\ar@{.>}[ur]^{q}\ar[r]_{\id}&C
}
\end{equation}
the left-hand side vertical arrow is a cofibration, and the right-hand side vertical arrow is an acyclic fibration. Therefore a dashed arrow $q\colon C\to \Cobar(\Bar(C))$ exists. One has
\begin{equation}
p\circ q=\id_C
\end{equation}
In particular, $q$ is a weak equivalence. 

Consider the map $$t^\prime=q^*\colon \Coh_\dg(\Cobar(\Bar(C)),D)\to \Coh_\dg(C,D)$$
It defines a quasi-isomorphism on all complexes of maps, because $q$ is a weak equivalence. It is also surjective on objects, because $p\circ q=\id$. Therefore, it defines an equivalence on the level of $H^0(-)$. 
\end{proof}

\begin{coroll}\label{corollstep0}
The following statents are true:
\begin{itemize}
\item[(i)] Let $C,C^\prime,D$ be small dg categories, $C,C^\prime$ cofibrant, and let $w\colon C\to C^\prime$ be a weak equivalence. Then the dg functor
$$
w^*\colon \Coh_\dg(C^\prime,D)\to\Coh_\dg(C,D)
$$
is a weak equivalence,
\item[(ii)] let $C,D,D_1$ be small dg categories, $C$ cofibrant, and let $t\colon D\to D_1$ be a weak equivalence. Then the dg functor
$$
t_*\colon\Coh_\dg(C,D)\to\Coh_\dg(C,D_1)
$$
is a weak equivalence. 
\end{itemize}
\end{coroll}
\begin{proof}
We prove (i), the proof of (ii) is similar. 

In the commutative diagram
\begin{equation}
\xymatrix{
\Coh_\dg(C^\prime,D)\ar[d]_{w^*}\ar[r]^i&\Coh_{A_\infty}(C^\prime,D)\ar[d]^{w^*}\\
\Coh_\dg(C,D)\ar[r]_{i}&\Coh_{A_\infty}(C,D)
}
\end{equation}
the horisontal arrows are weak equivalences by Lemma \ref{lemmastep0}, and the right vertical arrow is a weak equivalence by Proposition \ref{fquis}. Therefore, the left vertical arrow also is.
\end{proof}

{\it Step 1.}

\begin{prop}\label{propstep1}
Let $C,D$ be small dg categories, $C$ cofibrant. Then 
$\Coh_\dg(C,\hat{D})$ is a path object of $\Coh_\dg(C,D)$.
\end{prop}
\begin{proof}
We need to prove that in the diagram
\begin{equation}
\xymatrix{
&\Coh_\dg(C,\hat{D})\ar[d]^{(p_0,p_1)}\\
\Coh_\dg(C,D)\ar[ur]^{s_*}\ar[r]_{\Delta\hspace{12mm}}&\Coh_\dg(C,D)\times\Coh_\dg(C,D)
}
\end{equation}
the dg functor $s^*$ is a weak equivalence, and the dg functor $(p_0,p_1)$ is a fibration (here $s\colon D\to \hat{D}$ is the dg functor constructed in Section \ref{sectionpathtab}.

The dg functor $s$ is a weak equivalence, by Proposition \ref{proptabpath}. Therefore, $s^*$ is a weak equivalence, by Corollary \ref{corollstep0}(ii).

The proof that $(p_0,p_1)$ is a fibration is more involved; we essentially use Key-Lemma \ref{keylemma}. 

By Proposition \ref{propretract}, any cofibrant dg category $C$ is a retract of an $I$-cell complex (cofibrant) dg category $C_\sm$. It means that there are dg functors $C\xrightarrow{i}C_\sm\xrightarrow{\rho}C$ where $\rho\circ i=\id_C$ (in fact, $i$ and $\rho$ are weak equivalences). It gives the diagram
\begin{equation}
\xymatrix{
\Coh_\dg(C,\hat{D})\ar[r]^{\rho^*}\ar[d]_{(p_0,p_1)}&\Coh_\dg(C_\sm,\hat{D})\ar[r]^{i^*}\ar[d]_{(p_0^\prime,p_1^\prime)}&\Coh_\dg(C,\hat{D})\ar[d]^{(p_0,p_1)}\\
C\times C\ar[r]_{i\times i}&C_\sm\times C_\sm\ar[r]_{\rho\times \rho}&C\times C
}
\end{equation}
By Key-Lemma \ref{keylemma} we know that the middle vertical arrow $(p_0^\prime,p_1^\prime)$ is a fibration. On the other hand, the compositions of horizontal arrows are indentity maps. That is, the map $(p_0,p_1)$ is a retract of $(p_0^\prime,p_1^\prime)$.

It follows that $(p_0,p_1)$ is a fibration, by axiom M3) of a closed model category, see e.g. [Hir, Ch. 7.1].

\end{proof}

{\it Step 2.} 

Here we complete the proof of Theorem \ref{theorht}.

\begin{lemma}\label{lemmastep2}
Let small dg categories $C,D$ be cofibrant.
Then the dg category $C\sotimes D$ is cofibrant. 
\end{lemma}
\begin{proof}
Let firstly $C$ and $D$ be $I$-cell complexes. Denote by $\{f_i\}_{i\in I}$ and $\{g_j\}_{j\in J}$ morphisms in $C$ and $D$, correspondingly, which are their free generators. Then $C\sotimes D$ is an $I$-cell complex, with the set of free generating morphisms $$\{f_i\otimes\id_y\}_{i\in I, y\in \Ob(D)}\sqcup \{\id_x\otimes g_j\}_{j\in J, x\in \Ob(C)}\sqcup \{\varepsilon(f_i; s_1,\dots,s_n)\}_{i\in J, n\ge 1, s_1,\dots,s_n\in\Mor(D)}$$
In particular, $C\sotimes D$ is cofibrant, if $C$ and $D$ are $I$-cell complexes.

By Proposition \ref{propretract}, any cofibrant dg category is a retract of an $I$-cell complex dg category. 
Let $C,D$ be cofibrant, and $C_\sm, D_\sm$ be $I$-cell complex dg categories whose retracts $C$ and $D$ are. 
Then $C\sotimes D$ is a retract of $C_\sm\sotimes D_\sm$ which is cofibrant. Then $C\sotimes D$ itself is cofibrant, by axiom M3) of a closed model category, see e.g. [Hir, Ch. 7.1].
\end{proof}

Consider the adjunctions:
\begin{equation}\label{eqstep21}
\Fun_\dg(C\sotimes D,E)=\Fun_\dg(C,\Coh_\dg(D,E))
\end{equation}
\begin{equation}\label{eqstep22}
\Fun_\dg(C\sotimes D,\hat{E})=\Fun_\dg(C,\Coh_\dg(D,\hat{E}))
\end{equation}
which follow from \eqref{adj1}.

\begin{prop}\label{propstep2}
Let $C,D,E$ be small dg categories, with $C,D$ cofibrant. One has:
\begin{equation}\label{eqend1}
\Hot(C\sotimes D,E)=\Hot(C,\Coh_\dg(D,E))
\end{equation}
\end{prop}
\begin{proof}
The dg category $\Coh_\dg(D,\hat{E})$ is a path object of $\Coh_\dg(D,E)$ by Proposition \ref{propstep1}, because $D$ is  cofibrant.  
The dg category $C\sotimes D$ is cofibrant by Lemma \ref{lemmastep2}, because $C$ and $D$ are. Then the statement easily follows from \eqref{eqstep21}, \eqref{eqstep22}, and from the description of morphisms in the homotopy category via the right homotopy relation $\sim_R$, see Section \ref{sectionlrhot}.
\end{proof}

It is proven in [Fa2] that $\Coh_{A_\infty}(D,E)$ is isomorphic in $\Hot$ to the To\"{e}n category $\RRHom(C,D)$ of quasi-functors [To], for which one has
\begin{equation}
\Hot(C\otimes D,E)=\Hot(C,\RRHom(D,E))
\end{equation}
by [To, Corr. 6.4].

For $D$ cofibrant, the dg categories $\Coh_\dg(D,E)$ and $\Coh_{A_\infty}(D,E)$ are isomorphic as objects of $\Hot$, by Lemma \ref{lemmastep0}.

It follows from the Yoneda lemma that $C\sotimes D$ is isomorphic to $C\otimes D$ in $\Hot$.

\qed

\subsection{\sc A proof of Key-Lemma \ref{keylemma}}\label{sectionkl}
\subsubsection{\sc}\label{proofkl1}
Let us write down explicitely the objects and the morphisms of the dg category $\Coh_\dg(C,\hat{D})$. 

An object of $\Coh_\dg(C,\hat{D})$ is a dg functor $\Phi\colon C\to\hat{D}$. It is given by two dg functors $F,G\colon C\to D$, collection of closed degree 0 maps $\{\theta(X)\colon F(X)\to G(X)\}_{X\in C}$ which are isomorphisms in $H^0(D)$, and a collection of maps $\{h(f)\in \Hom^{-1}_D(F(X),G(Y))\}_{f\in C(X,Y), X,Y\in C}$, such that
\begin{equation}\label{kl1}
d(h(f))-(-1)^{|f|}h(d(f))=G(f)\circ \theta(X)-\theta(Y)\circ F(f)
\end{equation}
and
\begin{equation}\label{kl2}
h(f_2\circ f_1)=h(f_2)\circ F(f_1)+G(f_2)\circ h(f_1)
\end{equation}
\begin{equation}\label{kl3}
\xymatrix{
F(X)\ar[r]^{F(f)}\ar[d]_{\theta(X)}\ar[rd]^{h(f)}&F(Y)\ar[d]^{\theta(Y)}\\
G(X)\ar[r]_{G(f)}&G(Y)
}
\end{equation}
Note that \eqref{kl1} and \eqref{kl2} together just mean that $\Theta=(\theta, h,0,0,\dots)$ is a coherent natural transformation $\Theta\colon F\Rightarrow G$ of very special type: its components of degrees 2,3,... vanish. One also has that its 0-component $\{\theta(X)\colon F(X)\to G(X)\}_{X\in C}$ has degree 0 and is invertible in $H^0(D)$. 

Let $\Phi=(F,G,\Theta)$ and $\Phi_1=(F_1,G_1,\Theta_1)$ be two objects of $\Coh_\dg(C,\hat{D})$.
A morphism $\Psi_\ldot\colon\Phi\to\Phi_1$ in $\Coh_\dg(C,\hat{D})$ is given by three coherent natural transformations
$$
\Psi_1\colon F\Rightarrow F_1,\ \Psi_2\colon G\Rightarrow G_1,\Psi_3\colon F\Rightarrow G_1
$$
The boundary $d\Psi_\ldot$ is defined as
\begin{equation}\label{klsuper}
d(\Psi_1,\Psi_2,\Psi_3)=\big(d\Psi_1,\  d\Psi_2,\   d\Psi_3-\Psi_2\cup \Theta+\Theta_1\cup \Psi_1\big)
\end{equation}
Here $\Theta=(\theta,h,0,\dots)\colon F\Rightarrow G$ and $\Theta_1=(\theta_1,h_1,0,\dots)\colon F_1\Rightarrow G_1$ are the natural transformations introduced above. The sign $\cup$ denotes the ``vertical'' product in $\Coh_\dg(C,D)$.

One can compare this dg category $\Coh_\dg(C,\hat{D})$ with the path-object dg category $\widehat{\Coh_\dg(C,D)}$, we'll see these two dg categories are rather similar.

The dg category $\widehat{\Coh_\dg(C,D)}$ is a particular case of the Tabuada path-object category $\hat{X}$ of a small dg category $X$, see Section \ref{sectionpathtab}.
It has the following description. 

An object of $\widehat{\Coh_\dg(C,D)}$ is a triple $(F,G,\Xi)$ where $F,G\colon C\to D$ are dg functors, and $\Xi\colon F\Rightarrow G$ is a closed degree 0 coherent natural transformation, which defined an invertible morphism in $H^0(\Coh_\dg(C,D))$. One has:
\begin{lemma}\label{h0coh}
A degree 0 closed coherent natural transformation $\Xi\colon F\Rightarrow G\colon C\to D$ defines an invertible morphism in $H^0(\Coh_\dg(C,D))$ if an only if the closed degree 0 morphisms $$\{\Xi(X)\colon D(F(X),G(X))\}_{X\in\Ob(C)}$$ are all invertible in $H^0(D)$
\end{lemma}
\qed

A morphism $\Psi_\ldot\colon (F,G,\Xi)\to (F_1,G_1,\Xi_1)$ is given by a triple of natural transformations
$$
\Psi_1\colon F\Rightarrow F_1,\ \Psi_2\colon G\Rightarrow G_1,\Psi_3\colon F\Rightarrow G_1
$$
The differential $d\Psi_\ldot$ is given as
\begin{equation}\label{klsuperbis}
d(\Psi_1,\Psi_2,\Psi_3)=\big(d\Psi_1,\  d\Psi_2,\   d\Psi_3-\Psi_2\cup \Xi+\Xi_1\cup \Psi_1\big)
\end{equation}

Let us compare the two dg categories $\Coh_\dg(C,\hat{D})$ and $\widehat{\Coh_\dg(C,D)}$. It follows from Lemma \ref{h0coh} that for the coherent natural transformation $\Xi$ of type $\Theta$ (that is, with vanishing components in degrees $\ge 2$), the conditions ``are invertible in $H^0(-)$'' agree. The only difference is that, for the case of dg category $\Coh_\dg(C,\hat{D})$, the coherent natural tansformations $\Theta\colon F\Rightarrow G$ which figure out in the definition of {\it objects}, are those with vanishing components in degrees $\ge 2$. At the same time, those coherent natural transformations which figure out in {\it morphisms} (that is, $\Psi_1,\Psi_2,\Psi_3$), are totally the same for both dg categories.

\subsubsection{\sc}\label{proofkl2}
For the dg category $\widehat{\Coh_\dg(C,D)}$ we know that it is a path object of the dg category $\Coh_\dg(C,D)$, as a particular case of Proposition \ref{proptabpath}. 

For any two small dg categories $C,D$, the dg category $\Coh_\dg(C,D)$ contains a dg sub-category $\Coh_\dg^\lin(C,D)$ defined as follows. 

The objects of $\Coh_\dg^\lin(C,D)$ are all dg functors $F\colon C\to D$, that is, the same that the objects of $\Coh_\dg(C,D)$.

For two dg functors $F,G\colon C\to D$, the complex $\Coh_\dg^\lin(C,D)(F,G)$ is formed by the coherent natural transformations having only non-zero components in degrees 0 and 1. That is, it is formed by the coherent natural transformations $\Theta$, where 
$$
\begin{aligned}
\ &\Theta_0\in\prod_{X\in C}\Hom_D(F(X),G(X))\\
&\Theta_1\in \prod_{X,Y\in C}\Hom_\k\Big(\Hom_C(X,Y),\Hom_D(F(X),G(Y))\Big)\\
&\Theta_2,\Theta_3,\dots=0
\end{aligned}
$$
and such that for any morphism $X\xrightarrow{f}Y\xrightarrow{g}Z$ in $C$ one has:
\begin{equation}\label{eqder}
\Theta_1(g)F(f)-\Theta_1(gf)+G(g)\Theta_1(f)=0
\end{equation}
It is clear that one gets a subcomplex $\Coh_\dg^\lin(C,D)(F,G)\subset \Coh_\dg(C,D)(F,G)$.

\begin{lemma}\label{lemmaproofkl1}
Let $C,D$ be small dg categories, with $C$ an $I$-cell complex. Then, for any two dg functors $F,G\colon C\to D$, the imbedding of complexes 
$$
\Coh_\dg^\lin(C,D)(F,G)\hookrightarrow\Coh_\dg(C,D)(F,G)
$$
is a quasi-isomorphism. Consequently, the corresponding dg functor 
$$
i\colon \Coh_\dg^\lin(C,D)\to\Coh_\dg(C,D)
$$
is a quasi-equivalence.
\end{lemma}
\begin{proof}
It is well-known. One considers the case of a free dg category $C$ with 0 differential. Then its Hochschild cohomological complex $\Hoch^\udot(C,M)$ with coefficients in any bimodule $M$ is quasi-isomorphic to the sub-complex which is the total complex of the following two-terms bicomplex
$$
\Hoch^\udot_\lin(C,M)=\Big\{0\to \underset{\deg=0}{\prod_{X\in C}M(X)}\to\underset{\deg=1}{\prod_{X,Y\in C}\Hom_\k\big(C(X,Y), M(X,Y)\big)}\to 0\Big\}
$$
such that for the degree 1 elements the Leibniz rule analogous to \eqref{eqder} holds. 

Our case is corresponded to $M(X,Y)=D(F(X),G(Y))$. Next, when $C$ is an $I$-cell complex with non-zero differential, an elementary argument with spectral sequiences shows that the imbedding
$$
\Hoch^\udot_\lin(C,M)\hookrightarrow\Hoch^\udot(C,M)
$$
remains a quasi-isomorphism.
\end{proof}

\subsubsection{\sc }\label{proofkl3}
Turn back to our proof of Proposition \ref{proptabpath}. Starting with a solid arrow diagram in \eqref{f2lemma2}, we constructed $g\colon w\to z$, and $\xi\colon x\to z$ such that 
\begin{equation}\label{lemmax01}
d\xi=b\circ f-g\circ a
\end{equation}
which expresses that the constructed $\Phi\in\hat{C}(x\xrightarrow{f}y,w\xrightarrow{g}z)$ is closed. Then we constructed $\Phi^\prime\colon \hat{C}(w\xrightarrow{g}z,x\xrightarrow{f}y)$ such that 
\begin{equation}\label{lemmax02}
\begin{aligned}
\ &\Phi^\prime \Phi=\id({x\xrightarrow{f}y})+d\Gamma\\
&\Phi \Phi^\prime=\id({w\xrightarrow{g}z})+d\Gamma^\prime
\end{aligned}
\end{equation}

The main point in the following Lemma is that one can replace $g\colon w\to z$ to a cohomologous map $\bar{g}\colon w\to z$, in an arbitrary way, and then re-define the other dashed arrows in \eqref{f2lemma2} and \eqref{f2lemma6} appropriately (keeping the solid arrows fixed), such that that the statement of Proposition \ref{proptabpath} still holds. That is, there is more freedom in the choice of $g$ than one could reckon directly based on the proof of Proposition \ref{proptabpath}. 
\begin{lemma}\label{lemmax}
In the notations as above, replace $g\colon w\to z$ by a cohomologous $\bar{g}\colon w\to z$, 
\begin{equation}\label{eqx1}
\bar{g}=g+d(\varkappa)
\end{equation}
and define 
\begin{equation}\label{eqx2}
\bar{\xi}=\xi-\varkappa a
\end{equation}
Then define $\bar{\Phi}\in\hat{C}(x\xrightarrow{f}y,w\xrightarrow{\bar{g}}z)$ as 
\begin{equation}\label{eqx3}
\xymatrix{
x\ar[r]^{a}\ar[dr]^{\bar{\xi}}\ar[d]_{f}&w\ar[d]^{\bar{g}}\\
y\ar[r]_{b}&z
}
\end{equation}
where $a,b$ are the same as for $\Phi$. Then $\bar{\Phi}$ is closed morphism in $\hat{C}$, and there exists closed $\bar{\Phi}^\prime\in \hat{C}(w\xrightarrow{\bar{g}}z,x\xrightarrow{f}y)$ such that
\begin{equation}\label{eqx4}
\begin{aligned}
\ &\bar{\Phi}^\prime \bar{\Phi}=\id({x\xrightarrow{f}y})+d\bar{\Gamma}\\
&\bar{\Phi }\bar{\Phi}^\prime=\id({w\xrightarrow{\bar{g}}z})+d\bar{\Gamma}^\prime
\end{aligned}
\end{equation}
\end{lemma}

\begin{proof}
We define $a^\prime,b^\prime$ in the same way as in the proof of Proposition \ref{proptabpath}, as well as $h_x,h_y,h_z,h_y$. In particular, \eqref{f2lemma4} holds.

The closedness of $\bar{\Phi}$ amounts to the equation
\begin{equation}\label{eqx5}
d\bar{\xi}=bf-\bar{g}a
\end{equation}
which follows directly from \eqref{lemmax01}, \eqref{eqx1}, and \eqref{eqx2}.

Define $\bar{\Phi}^\prime\in\hat{C}(w\xrightarrow{\bar{g}}z,x\xrightarrow{f}y)$ as
\begin{equation}\label{eqx6}
\xymatrix{
w\ar[r]^{a^\prime}\ar[rd]^{\bar{\mu}}\ar[d]_{\bar{g}}&x\ar[d]^{f}\\
z\ar[r]_{b^\prime}&
}
\end{equation}
where
\begin{equation}\label{eqx7}
\bar{\mu}=\mu+b^\prime\varkappa
\end{equation}
where $\mu$ is as in \eqref{f2lemma7}. Then the closedness of $\bar{\Phi}^\prime$
\begin{equation}\label{eqx8}
d\bar{\mu}=b^\prime\bar{g}-fa^\prime
\end{equation}
follows directly from \eqref{f2lemma8}, \eqref{eqx1}, and \eqref{eqx7}.

It remains to prove \eqref{eqx4}.

The composition $\bar{\Phi}^\prime\bar{\Phi}$ is given by
\begin{equation}\label{eqx9}
\xymatrix{
x\ar[d]_{f}\ar[rr]^{a^\prime a}\ar[rrd]^{\ell_1}&&x\ar[d]^{f}\\
y\ar[rr]_{b^\prime b}&&y
}
\end{equation}
where
\begin{equation}\label{eqx10}
\ell_1=b^\prime \bar{\xi}+\bar{\mu}a
\end{equation}
We have:
\begin{equation}\label{eqx11}
\ell_1=(b^\prime \xi-b^\prime\varkappa a)+(\mu a+b^\prime \varkappa a)
\end{equation}
We see that the counter-terms $-b^\prime \varkappa a$ and $b^\prime \varkappa a$ are mutually cancelled, and finally
\begin{equation}\label{eqx12}
\bar{\Phi}^\prime\bar{\Phi}=\Phi^\prime\Phi
\end{equation}
Therefore,
\begin{equation}\label{eqx13}
\bar{\Phi}^\prime\bar{\Phi}=\id(x\xrightarrow{f}y)+d\bar{\Gamma}
\end{equation}
with $\bar{\Gamma}=\Gamma$, see \eqref{f2lemma10}, \eqref{f2lemma11}. 

The composition $\bar{\Phi}\bar{\Phi}^\prime$ is equal to
\begin{equation}\label{eqx14}
\xymatrix{
w\ar[d]_{\bar{g}}\ar[rr]^{aa^\prime}\ar[rrd]^{\ell_2}&&w\ar[d]^{\bar{g}}\\
z\ar[rr]_{bb^{\prime}}&&z
}
\end{equation}
where
\begin{equation}\label{eqx15}
\ell_2=\bar{\xi}a^\prime+b\bar{\mu}=(\xi a^\prime+b\mu)+(-\varkappa a a^\prime +b b^\prime \varkappa)
\end{equation}
Here the counter-terms are not cancelled, but the vertical map $g$ is also replaced by $\bar{g}=g+d\varkappa$.

One has:
\begin{equation}
\ell_2=(\xi a^\prime+b\mu)+(-\varkappa dh_w+dh_z\varkappa)=(\xi a^\prime+b\mu)+d(\varkappa h_w+h_z\varkappa)+(-(d\varkappa) h_w+h_z d\varkappa)
\end{equation}
Define $\bar{\Gamma}^\prime$ as
\begin{equation}
\xymatrix{
w\ar[r]^{h_w}\ar[d]_{\bar{g}=g+d\varkappa}\ar[dr]^{\delta}&w\ar[d]^{\bar{g}=g+d\varkappa}\\
z\ar[r]_{h_z}&z
}
\end{equation}
where
\begin{equation}
\delta=\epsilon^\prime+\varkappa h_w+h_z\varkappa
\end{equation}
see \eqref{eqstrange}.

Then one has
\begin{equation}\label{eqx16}
\bar{\Phi}\bar{\Phi}^\prime=\id(w\xrightarrow{\bar{g}}z)+d\bar{\Gamma}^\prime
\end{equation}

\end{proof}

\subsubsection{\sc }\label{proofkl4}
We prove Key-Lemma \ref{keylemma}.

In Section \ref{proofkl1} we have seen that the complex 
$$
(p_0\times p_1)^{-1}(F,G)\subset\Coh_\dg(C,\hat{D})
$$
is identified with $\Coh_\dg^\lin(C,D)(F,G)$, and that the functor $\Coh_\dg(C,\hat{D})\subset \widehat{\Coh_\dg(C,D)}$ is full.

In Lemma \ref{lemmaproofkl1} we have seen that, for $C$ an $I$-cell complex, the imbedding  $\Coh_\dg^\lin(C,D)(F,G)\hookrightarrow\Coh_\dg(C,D)(F,G)$ is a quasi-isomorphism.

Then we can apply Lemma \ref{lemmax}, taking for $\bar{g}$ a cohomologous linear coherent natural transformation. It shows that, for $C$ an $I$-cell complex, $\Coh_\dg(C,\hat{D})$ is a path object for the dg category $\Coh_\dg(C,D)$.

Key-Lemma \ref{keylemma} is proven.

\qed

\comment
\appendix

\section{\sc A construction of the set $\Hot(C,D)$ where $C$ is semi-free, via the convolution $L_\infty$ algebra}
In the Appendix, we recall the construction of the $\Hom$-set in the homotopy category $\Hot(C,D)$, where $C$ semi-free, via the Maurer-Cartan elements in  the ``convolution $L_\infty$ algebra''. 
The morphisms $C\to D$ are corresponded to the Maurer-Cartan elements in the convolution $L_\infty$ algebra, while the homotopy relation on them is given by the ``gauge action'' of degree 0 elements in the convolution $L_\infty$ algebra, in the sense of Deligne groupoid [GM].

After reminder on the homotopy relation in a closed model category, we consider at first the case of the closed model category of associative algebras over $\k$. After that, we discuss the homotopy relation in the closed model category of small dg categories over $\k$.

\subsection{\sc Reminder on the homotopy relations in a closed model category}\label{sectionlrhot}
Let $\mathscr{M}$ be a closed model category [Q], [GJ, Ch. II, Section 1], [Hir, Ch. 7]. Recall that one derives a {\it homotopy relation} on $\Hom_\mathscr{M}(C,D)$ from the basic axioms of a closed model category. There are two such relations. 

The first one $\sim_L$ uses {\it a cylinder object} of $C$, and is an equivalence relation when $C$ is cofibrant, in which case the realation $f\sim_L g$ does not depend on the choice of a cylinder object. 

The second one $\sim_R$ uses {\it a path object} of $D$, and is an equivalence relation when $D$ is fibrant, in which case the relation $f\sim_R g$ does not depend on the choice of a path object.

Moreover, when $C$ is cofibrant and $D$ is fibrant, $f\sim_L g$ holds iff $f\sim_R g$, see Proposition \ref{proplr} below.

{\it A cylinder object} of an object $C$ of $\mathscr{M}$ is a commutative triangle 
\begin{equation}
\xymatrix{
C\sqcup C\ar[dr]^{\nabla}\ar[d]_{i}\\
\tilde{C}\ar[r]_{\sigma}&C
}
\end{equation}
where $\nabla\colon C\sqcup C\to C$ is the canonical map defined from the identity map $C\to C$ on each summand, $i$ is a cofibration, and $\sigma$ is a weak equivalence. 

The cylinder object always exists but is not unique. 

The {\it left homotopy relation} $f\sim_L g$, for $f,g\colon C\to D$, is defined for a given choice of a cylinder object of $C$. It is defined as a commutative diagram
\begin{equation}
\xymatrix{
C\sqcup C\ar[d]_{i}\ar[dr]^{(f,g)}\\
\tilde{C}\ar[r]_h&D
}
\end{equation}
where $(f,g)$ is the map defined via $f$ and $g$ on the summands $C$, and the data $C\sqcup C\xrightarrow{i}\tilde{C}$ comes from some choice of a cylinder object.

In general, $\sim_L$ is not an equivalence relation on $\Hom(C,D)$, but it is when $C$ is cofibrant. See Proposition \ref{proplr} below. 

Recall also the path objects and the right homotopy relation.

Let $\mathscr{M}$ be a closed model category, $D$ an object of $\mathscr{M}$.

A {\it path object} of $D$ is a commutative triangle
\begin{equation}\label{eqpath}
\xymatrix{
&\hat{D}\ar[d]^{p=(p_0,p_1)}\\
D\ar[ur]^s\ar[r]_{\Delta\hspace{12pt}}&D\times D
}
\end{equation}
where $\Delta$ is the diagonal map, $s$ is a weak equivalence, and $p$ is a fibration.  The path objects always exist but are not unique.

Assume a path object of $D$ is chosen. One says that $f,g\in\Hom(C,D)$ are {\it right homotopic} with respect to the path object if there is a commutative triangle
\begin{equation}\label{equationrhot}
\xymatrix{
&\hat{D}\ar[d]^{p=(p_0,p_1)}\\
C\ar[ur]^{\theta}\ar[r]_{(f,g)\hspace{12pt}}&D\times D
}
\end{equation}

In general, $\sim_R$ is not an equivalence relation on $\Hom(C,D)$, but it is, if $D$ is fibrant, see Proposition \ref{proplr} below. 

\begin{prop}\label{proplr}
Let $\mathscr{M}$ be a closed model category. The following statements are true:
\begin{itemize}
\item[(i)] For a cofibrant object $C$, the left homotomy relation $\sim_L$ on $\Hom(C,D)$ is an equivalence relation. Moreover, in this case, if $f\sim_L g$ for one choice of the cylinder object of $C$, then $f\sim_L g$ for any other choice,
\item[(ii)] For a fibrant object $D$, the right homotopy relation $\sim_R$ on $\Hom(C,D)$ is an equivalence relation. Moreover, in this case, if $f\sim_R g$ for one choice of the path object of $D$, then $f\sim_R g$ for any other choice,
\item[(iii)] Let $C$ be cofibrant and $D$ be fibrant. Then on $\Hom(C,D)$ both homotopy relations coincide: $f\sim_L g\Leftrightarrow f\sim_R g$.
\end{itemize}
\end{prop}
See [GJ, Cor. II 1.9] for a proof.

\qed

The {\it homotopy category} $\Ho(\mathscr{M})$ is defined as a category whose objects are the same as the objects of $\mathscr{M}$, and whose morphisms $\Ho(\mathscr{M})(C,D)$ is $\mathscr{M}(RQ(C),RQ(D))/\sim_L$, or equally, $\mathscr{M}(RQ(C),RQ(D))/\sim_R$, where $RQ(X)$ is a fibrant cofibrant replacement of the object $X$.
The reader is refered to [GJ, Ch. II, Section 1] for more detail. 

\subsection{\sc The Deligne groupoid}
Here we recall the Deligne groupoid  $D(L)$, associated with an $L_\infty$ algebra $L$. We also recall its simplicial interpretation given in [H2] and [G].

Let $\g$ be an $L_\infty$ algebra over a field of characteristic 0, denote by $L_2,L_3,\dots$ its Taylor components. 

The objects of the groupoid $D(\g)$ are the Maurer-Cartan elements, that is the elements $x\in \g^1$ such that
\begin{equation}\label{appa1}
dx +\frac12L_2(x,x)+\frac16L_3(x,x,x)+\dots+\frac1{n!}L_n(x,x,\dots,x)+\dots=0
\end{equation}
(provided the convergence of the above series).

The morphisms are given by elements $h\in\g^0$, as follows. Each $x\in \g^0$ defines a vector field on $\g^1$:
\begin{equation}\label{appa2}
\dot{x}=dh+L_2(h,x)+\frac12L_3(h,x,x)+\dots+\frac1{(n-1)!}L_n(h,x,\dots,x)+\dots
\end{equation}
Denote this vector field by $v(h)$. 

One has:
\begin{lemma}
Let $\g$ be an $L_\infty$ algebra. The following statement are true:
\begin{itemize}
\item[(1)]
The vector fields $v(h)$, $h\in \g^0$, preserve the ``subvariety'' of Maurer-Cartan elements,
\item[(2)] 
The map $\g^0\to\Vect(\g^1)$, $h\mapsto v(h)$, is a map of Lie algebras. 
\end{itemize}
\end{lemma}

One can formally integrate $v(h)$, to a map $\exp(v(h))\colon \g^1\to \g^1$, provided the convergence of the exponent. 
By definition, $h\in \g^0$ gives a morphism $f_h\colon x\to \exp(v(h))(x)$ in the Deligne groupoid $D(\g)$. The composition of two of such morphisms is once again a morphism of this type, because $\g^0$ is a Lie subalgebra and due to the Campbell-Baker-Hausdorff formula, provided its convergence.

The problem of convergence is essential. In deformation theory, one takes an artinian commutative algebra $a$, and considers the $L_\infty$ algebra $\g\otimes a_+$, where $a_+$ is the nilpotent kernel of the augmentation map, $a_+^N=0$ for some $N$. Then all convergence problems are automatically solved.

In our case, $\g=\Conv(C,D)$ is pro-nilpotent by construction. It guarantees that all formulas we deal with, such as \eqref{appa1}, \eqref{appa2}, the formula for $\exp(v(h))$, and the CBH formula, do converge. Therefore, the Deligne groupoid $D(\Conv(C,D))$ is well-defined.

With a groupoid $G$ one associates its set of connected components $\pi_0(G)$.

Let $\g$ be an $L_\infty$ algebra, $D(\g)$ its Deligne groupoid. The set $\pi_0(D(\g))$ has a description as $\pi_0(\MC_\ldot(\g))$ of a simplicial set $\MC_\ldot(\g)$. It is due to Hinich [H2] and Getzler [G], see also [DR].

Denote by $\Omega^\udot_n$ the commutative dg algebra of polynomial differential forms on $\Delta_n$:
$$
\Omega_n^\udot=\k[x_0,\dots,x_n, dx_0,\dots,dx_n]/(x_0+\dots+x_n-1, dx_0+\dots+dx_n)
$$
where $\deg x_i=0, \deg dx_i=1$, and $d(x_i)=dx_i, d(dx_i)=0$. 

Taking $\g_n:=\g\hat{\otimes} \Omega^\udot_n$, one gets a simplicial $L_\infty$ algebra $\g_\ldot$. Then define 
$$
\MC_n(\g):=\MC(\g_n)
$$
It is a simplicial set. It is known to be a Kan simplicial set, which is a higher generalization of the Deligne groupoid. 
We need here only its 1-categorical counterpart.

One has:
\begin{prop}[{[H2],[G]}]\label{prophg}
For an $L_\infty$ algebra $\g$ over a field of characteristic 0, one has:
$$
\pi_0(\MC_\ldot(\g))=\pi_0(D(\g))
$$
\end{prop}
\begin{proof}
We only sketch the argument, referring the reader to [H2], [G], and [DR] for more detail.

In fact, to compute $\pi_0(\MC_\ldot(\g))$ one needs only $\MC_0(\g)=\MC(\g)$ and $\MC_1(\g)=\MC(\g\hat{\otimes} \Omega_1^\udot)$.

Let $h\in \g^0$, and let $f(t)\in \g^1\hat{\otimes}\k[t]$ be such that
\begin{equation}\label{appa10}
K=-h\otimes dt+f(t)\in \MC(\g\hat{\otimes}\Omega_1^\udot)
\end{equation}
(we identify $\Omega_1^\udot=\k[t,dt]$ with $d(t)=dt$, where $t=x_0-x_1$). 

We assume that $h$ does not depend on $t$; the general case is reduced to this one.

The Maurer-Cartan equation for $K$ results in two equations, which are vanishing of $(dt)^1$ and $(dt)^0$ components. They are:
\begin{equation}\label{appa11}
-{d_\g}h+\dot{f}(t)-L_2(h,f(t))-\frac12L_3(h,f(t),f(t))-\dots=0
\end{equation}
and
\begin{equation}\label{appa12}
d_\g(f(t))+\frac12L_2(f(t),f(t))+\frac16L_3(f(t),f(t),f(t))+\dots=0
\end{equation}
Equation \eqref{appa12} means that $f(t)$ is a Maurer-Cartan element, or, equivalently, that the specification $f(a)$ is a Maurer-Cartan element, for any $a\in \k$. Equation \eqref{appa11} just means that $f(t)$ is a solution of the equation \eqref{appa2}.

\end{proof}

\subsection{\sc An equivalent version}
In Proposition \ref{prophg}, the morphisms in the Deligne groupoid of an $L_\infty$ algebra $\g$ are interpretated via the Maurer-Cartan elements in the bigger $L_\infty$ algebra $\g\otimes \k[t,dt]$. There are two maps of $L_\infty$ algebras 
$$
p_0,p_1\colon \g\otimes \k[t,dt]\to \g
$$
given by evaluation maps at $t=0$ and at $t=1$, correspondingly. 

It gives rise to two maps 
$$
p_0,p_1\colon \MC(\g\otimes \k[t,dt])\to\MC(\g)
$$
and Proposition \ref{prophg} says that two Maurer-Cartan elements $\alpha_0,\alpha_1\in \MC(\g)$ are connected by a morphism in the eligne grouppoid $D(\g)$ if and only if there is a Maurer-Cartan element $A\in \MC(\g\otimes \k[t,dt])$ such that 
$$
p_0(A)=\alpha_0,\ p_1(A)=\alpha_1
$$

Here we replace $\g\otimes \k[t,dt]$ by a ``smaller'' quasi-isomorphic $L_\infty$ algebra, what makes it possible to find an equivalent criterium when $\alpha_0$ and $\alpha_1$ are connected by a morphism in the Deligne groupoid $D(\g)$.

Consider the following commutative dg algebra $\varkappa$.

As a graded vector space, $\varkappa$ has two-dimensional component in degree 0, and 1-dimensional component in degree 1 (all other components vanish).

Denote by $x_0$ and $x_1$ degree 0 elements and by $x_{01}$ degee 1 elements. Define the differential as 
$$
d(x_0)=d(x_1)=x_{01},\ d(x_{01})=0
$$
Define a commutative algebra structure as
$$
x_0^2=x_0,\  x_1^2=x_1, \ x_0x_1=0,\  2x_0x_{01}=x_{01},\ 2x_1x_{01}=x_{01},\  x_{01}^2=0
$$
\begin{lemma}
Denote by $i\colon \varkappa\to \k[t,dt]$ the map sending $x_0\to t$
\end{lemma}
\subsection{\sc The case of associative dg algebras}
\subsubsection{\sc The convolution $L_\infty$ algebra}
Recall that a dg associative algebra $A$ over $\k$ is called {\it semi-free} if there is an exhaustive ascending filtration $$A_0\subset A_1\subset A_2\subset \dots$$
where each $A_i$ is a dg associative algebra, the underlying graded algebra of $A_{i+1}$ is freely generated by $A_i$ and a set of elements $V_{i+1}=\{f_{i+1,1},\dots,f_{i+1,n_{i+1}}\}$, and 
$$
d(f_{i+1,k})\subset A_i \text{  for any $k$}
$$
(the latter codition for $i=-1$ means that $d(f_{0,k})=0$ for any $k$).

Similarly, let in the above definition $A_0$ be some dg associative algebra (with possibly non-zero differential), and one has an exhaustive ascending filatration $A_0\subset A_1\subset A_2\subset\dots$ on $A$, such that each $A_i$ is a dg associative algebra, the underlying graded algebra of $A_{i+1}$ is freely generated by $A_i$ and a set of elements $V_{i+1}=\{f_{i+1,1},\dots,f_{i+1,n_{i+1}}\}$, and 
$$
d(f_{i+1,k})\subset A_i \text{  for any $k$ and for $i\ge 0$}
$$
(that is, the condition for $i=-1$ is dropped). Then the imbedding $$i\colon A_0\to A$$ is called {\it a standard cofibration}.

Recall that the category of dg associative algebras over a field $\k$ admits a closed model structure due to Hinich [H1]. Its weak equivalences are quasi-isomorphisms, fibrations are component-wise surjective maps, and cofibrations are the maps satisfying the left lifting property with respect to acyclic fibrations, see [H1, Theorem 2.2.1]. One can prove that a map in this closed model structure is a cofibration if it is a retract of a standard cofibration, see [H1, Remark 2.2.5].

Let $C$ be a semi-free dg associative algebra, $V$ its space of generators. Then $V[1]$ becomes an $A_\infty$ coassociative coalgebra. Indeed, this $A_\infty$ coalgebra structure is given by the differential on the free algebra $C=T_+(V)=T_+(V[1][-1])$, making it a dg associative algebra. (Here $T_+(V)=\Ker(\varepsilon\colon T(V)\to\k)$ is the kernel of the augmentation map).

Let $D$ be an arbitrary associative dg algebra, and $f\colon C\to D$ a dg algebra map. It is given by its restriction to generators $V$, which is a map $f_0\colon V\to D$, compatible with the differential. The compatibility with the differential admits the following interpretation. 

Set $A=\Hom_\k(V[1],D)$. It is a non-unital $A_\infty$ algebra. Indeed, the Taylor components
$$
m_n\colon A^{\otimes n}\to A[2-n]
$$
are given as ``convolutions'':
\begin{equation}\label{convprod}
V[1]\xrightarrow{\delta_{n-1}}V[1]^{\otimes n}[n-2]\xrightarrow{t_1\otimes\dots\otimes t_n}D^{\otimes n}[n-2]\xrightarrow{m}D[n-2]
\end{equation}
where $\delta_{n-1}$ is the $(n-1)$-st Taylor component of the $A_\infty$ coalgebra $V[1]$, $t_1,\dots,t_n\in A$, and $m$ is given by the iterated product in $D$. 

The skew-symmetrisation of each Taylor component makes $A$ an $L_\infty$ algebra. We denote this $L_\infty$ by $\Conv(C,D)$.

\subsubsection{\sc }
Let $C,D$ be two associative dg algebras, with $C$ semi-free. Recall that the Maurer-Cartan elements in $\Conv(C,D)$ are in 1-to-1 correspondence with maps of dg algebras $C\to D$: 
\begin{equation}
\mathrm{MC}(\Conv(C,D))=\Alg_\dg(C,D)
\end{equation}

\begin{theorem}
Let $C,D$ be dg associative algebras, with $C$ semi-free. Let $f,g\in\Alg_\dg(C,D)$. Then $f\sim_R g$ if and only if the Maurer-Cartan elements $f$ and $g$ are connected by a morphism (corresponded to $v(h)$, $h\in \Conv(C,D)^0$), in the Deligne groupoid $D(\Conv(C,D))$. 
\end{theorem}
\begin{proof}
Construct explicitely a path object of $D$, as follows.

Consider the commutative dg algebra $\k[t,dt]$ where
\begin{equation}\label{ktdt}
\deg t=0,\deg (dt)=1, d(t)=dt, d(dt)=0
\end{equation}
Consider $\hat{D}=D\otimes \k[t,dt]$.
Define two maps $p_0,p_1\colon \hat{D}\to D$ as $p_i(x)=x|_{t=i, dt=0}$. The maps $p_0,p_1$ are maps of dg associative algebras, and they give $(p_0,p_1)\colon \hat{D}\to D\times D$ (recall that the underlying vector space of the direct product $A_1\times A_2$ of two (dg) algebras $A_1$ and $A_2$ is $A_1\oplus A_2$, and the product is $(a,b)\cdot (c,d)=(ac,bd)$). The map $\Delta\colon D\to D\times D$ is $\Delta(d)=(d,d)$, and the map $s\colon D\to\hat{D}$ is $s(d)=d\cdot 1+ 0\cdot dt$ (that is, $s(d)$ is the constant 0-form equal to $d$). It is clear that the diagram \eqref{eqpath} commutes. Moreover, $s$ is a quasi-isomorphism (that is, a weak equivalence), because $\k[t,dt]$ is quasi-isomorphic to $\k$ over a field of characteristic 0, and $(p_0,p_1)$ is a fibration (that is, is component-wise surjective).  It follows that $\hat{D}$ defines a path object of $D$, see Section \ref{sectionlrhot}.

Then $f,g\in\Hom(C,D)$ are right-homotopic with respect to $\hat{D}$ if there is a map $\theta\colon C\to\hat{D}$ such that the diagram \eqref{equationrhot} commutes.

The map $\theta$ defines a Maurer-Cartan element in $\Conv(C,D\otimes \k[t,dt])=\Conv(C,D)\hat{\otimes}\k[t,dt]$.
This Maurer-Cartan element is given by its components with $1$ and $dt$:
$$
\theta=-h(t)dt+F(t)
$$
where $h(t)\in\g^0\hat{\otimes}\k[t]$, $F(t)\in\g^1\hat{\otimes}\k[t]$, $\g=\Conv(C,D)$.
A priori $h(t)$ depends on $t$. Denote by $h_0,h_1,\dots$ the Taylor coefficients of $h$. 
One easily sees that for each $i$, $\theta_i=-h_i\otimes t^idt+F(t)$ is a Maurer-Cartan element.
In particular, $\theta_0=-h_0dt+F(t)$ is a Maurer-Cartan element. The commutativity of the diagram \eqref{equationrhot} implies that $F(0)=f, F(1)=g$.

Proposition \ref{prophg} gives the result.

\end{proof}

\subsection{\sc The case of small dg categories}
\subsubsection{\sc The closed model structure}\label{tabcmc}
A {\it graph} $\Gamma$ of graded vector categories has a set of objects $\Ob\Gamma$, and for each $x,y\in \Ob\Gamma$, a graded vector space $\Gamma(x,y)$.

A small dg category $C$ is called {\it semi-free} if there is an ascending filtration of dg categories $C_0\subset C_1\subset C_2\subset\dots$ such that all $C_i$ have $\Ob C$ as their objects, the category $C_{i+1}$ is freely generated by $C_i$ and a graph $\Gamma_{i+1}$ with objects $\Ob C$, and such that for any $f\in\Gamma_{i+1}(x,y)$, $df\in C_i$. In particular, the category $C_0$ is freely generated by a graph $\Gamma_0$, and for any $f\in C_0(x,y)$ one has $df=0$. 

Dropping the condition that $C_0$ is free dg category with 0 differential, one gets a {\it standard cofibration} $i\colon C_0\to C$. 

Recall that the category $\Cat_\dg(\k)$ of small dg categories over $\k$ admits a cofibrantly generated closed model structure, described in [Tab1]. Its weak equivalences are quasi-equivalences of dg categories, and fibrations are described as follows:

A dg functor $F\colon C\to D$ is called a {\it fibration} if (a) for any two objects $x,y\in C$, the map $C(x,y)\to D(F(x),F(y))$ induced by $F$, is a fibration of complexes (that is, a component-wise surjection), and (b) for any $x\in C$ and any isomorphism $v\colon [F](x)\to y^\prime$ in $H^0(D)$ there exists an isomorphism $u\colon x\to y$ in $H^0(C)$ such that $[F](u)=v$. In particular, any object is fibrant.

The cofibrations are defined as those morphisms which have the left lifting property with respect to the acyclic fibrations. 

One has the following explicit description of the cofibrations.
\begin{prop}
Any cofibration in the closed model structure on $\Cat_\dg(\k)$ is a retract of a standard cofibration. Any cofibrant object is a retract of a semi-free dg category. 
\end{prop}
\begin{proof}
It follows from the general description of cofibrations in a cofibrantly generated closed model category, given in [Hir, Prop. 11.2.1 (1)], and from the description of generating cofibrations, given in [Tab1, Th. 1.8].
\end{proof}

\subsubsection{\sc The convolution $L_\infty$ algebra}
Let $C,D$ be small dg categories, with $C$ semi-free. The dg functors $\Fun_\dg(C,D)$ can be described, via the convolution dg Lie algebra, as follows.

Denote by $\Gamma$ the graph of generators for $C$. As a graded $\k$-linear category, $C$ is the free category generated by the graph $\Gamma$. Then $\Gamma[1]$ becomes an $A_\infty$ cocategory. 

Let $\Phi\in\Sets(\Ob C,\Ob D)$. Define a non-unital $A_\infty$ algebra 
\begin{equation}
\Conv_\Phi(C,D):=\bigoplus_{X,Y\in \Ob C}\Hom_\k(\Gamma(X,Y)[1],D(\Phi(X),\Phi(Y))
\end{equation}
The Taylor components of this $A_\infty$ algebra are defined via the convolution products, as in \eqref{convprod}.

Finally, skew-symmetrising the Taylor components of the $A_\infty$ algebra $\Conv_\Phi(C,D)$, we get an $L_\infty$ algebra. 
By abuse of notations, we denote this $L_\infty$ algebra also by $\Conv_\Phi(C,D)$. 

One easily has:
\begin{lemma}
Let $F\colon C\to D$ be a dg functor whose action on objects is given by a map $\Phi\colon \Ob C\to\Ob D$. Then $F$ defines a Maurer-Cartan element in the $L_\infty$ algebra $\Conv_\Phi(C,D)$. Conversely, any Maurer-Cartan element in $\Conv_\Phi(C,D)$ gives a dg functor $F\colon C\to D$ whose action on objects is $\Phi$.
\end{lemma}
\qed

\subsubsection{\sc The homotopy relation and its interpretation via the convolution $L_\infty$ algebra}
Let $D$ be a small dg category over $\k$. Recall our dg commutative algebra $\k[t,dt]$, see \eqref{ktdt}.
Consider $\hat{D}:=D\otimes \k[t,dt]$. It is a dg category having the same objects that $D$, and $\hat{D}(x,y)=D(x,y)\otimes \k[t,dt]$.

There is a diagram of dg categories and dg functors
\begin{equation}\label{eqpathdg}
\xymatrix{
&\hat{D}\ar[d]^{p=(p_0,p_1)}\\
D\ar[ur]^s\ar[r]_{\Delta\hspace{12pt}}&D\times D
}
\end{equation}
making $\hat{D}$ a path object of $D$. The maps $p_0,p_1$ are evaluation maps $dt\to 0$, and $t\to 0$ or $t\to 1$, correspondingly. The maps $s$ and $\Delta$ are clear. Recall that the dg category $A_1\times A_2$ has the set of objects $\Ob A_1\times \Ob A_2$, and $\Hom_{A_1\times A_2}(x\times y, x_1\times y_1)=\Hom_{A_1}(x,x_1)\oplus\Hom_{A_2}(y,y_1)$.
The map $s$ is clearly a weak equivalence. 

The dg functor $(p_0,p_1)$ is given $x\to (x,x)$ on objects, and $D(x,y)\otimes f(t,dt)\to (D(x,y)f(0,0)\oplus D(x,y)f(1,0))$ on morphisms.
We need to show that this dg functor is a fibration, see Section \ref{tabcmc}. Condition (a) is clearly fulfilled. For condition (b), note that $H^0(\hat{D})=H^0(D)$, and $H^0(D\times D)=H^0(D)\times H^0(D)$. The functor $[(p_0,p_1)]$ is the diagonal functor. 

\endcomment

\bigskip

{\small
\noindent {\sc Universiteit Antwerpen, Campus Middelheim, Wiskunde en Informatica, Gebouw G\\
Middelheimlaan 1, 2020 Antwerpen, Belgi\"{e}}}

\vspace{1mm}

{\small
\noindent{\sc Laboratory of Algebraic Geometry,
National Research University Higher School of Economics,
Moscow, Russia}}

\bigskip

\noindent{{\it e-mail}: {\tt Boris.Shoikhet@uantwerpen.be}}

\end{document}